\documentclass[11pt,leqno,oneside]{amsart}
\usepackage[colorlinks, linkcolor=teal, citecolor=violet]{hyperref}
\usepackage{amsmath,amsfonts,amssymb,amsthm,enumerate,mathtools,comment}
\usepackage{a4wide,color,framed,mathscinet,doi}
\usepackage[mathcal]{euscript}
\usepackage[a4paper, left=2cm, right=2cm, top=3cm, bottom=3cm]{geometry}
\usepackage[numbers,sort&compress]{natbib}
\usepackage[table]{xcolor}
\usepackage[textsize=tiny,textwidth=3.2cm,colorinlistoftodos]{todonotes}
\usepackage{multirow}
\usepackage{enumitem}
\setlist[enumerate]{label={\rm(\roman*)},leftmargin=6ex}
\newcommand{\R}{\mathbb{R}}
\newcommand{\rn}{\mathbb{R}^n}
\newcommand{\N}{\mathbb{N}}

\newcommand{\RR}{\mathcal{R}}
\renewcommand{\d}{{\fam0 d}}

\newcommand{\essinf}{\operatornamewithlimits{ess\,inf}}

\newcommand{\rpar}{r}

\let\tilde\widetilde

\DeclarePairedDelimiter\nrm{\lVert}{\rVert}

\def\bmc{\alpha}

\usepackage{xspace}
\makeatletter
\DeclareRobustCommand\onedot{\futurelet\@let@token\@onedot}
\def\@onedot{\ifx\@let@token.\else.\null\fi\xspace}
\def\eg{e.g\onedot} 
\def\ie{i.e\onedot} 
\def\cf{cf\onedot} 
\makeatother

\newtheoremstyle{MyPlain}{}{}{\itshape}{}{\bfseries}{.}{5pt plus 4pt minus 3pt}{\thmname{#1}\thmnumber{ #2}\thmnote{ \textbf{[#3]}}}
\theoremstyle{MyPlain}
\newtheorem{theorem}{Theorem}[section]
\newtheorem{lemma}[theorem]{Lemma}

\newtheorem{corollary}[theorem]{Corollary}

\newtheoremstyle{MyRemark}{}{}{\upshape}{}{\bfseries}{.}{5pt plus 1pt minus 1pt}{}
\theoremstyle{MyRemark}
\newtheorem*{definition}{Definition}

\newtheorem{remark}[theorem]{Remark}

\numberwithin{equation}{section}

\expandafter\let\expandafter\oldproof\csname\string\proof\endcsname
\let\oldendproof\endproof
\renewenvironment{proof}[1][\proofname]{%
  \oldproof[{{\bf #1.}}]%
}{\oldendproof}

\makeatletter
\def\paragraph{\medskip\@startsection{paragraph}{4}%
  \z@\z@{-\fontdimen2\font}%
  {\normalfont\bfseries}}
\makeatother

\newcommand{\tff}{\tilde{\varphi}}
\def\en{\mathbb N}
\def\vint{-\kern-11pt\int}
\newcommand{\RRR}{\mathcal R}

\DeclareMathOperator{\spt}{supp}

\begin{document}

\title{Maximal noncompactness of embeddings into Marcinkiewicz spaces}

\begin{abstract}
We develop a new functional-analytic technique for investigating the degree of noncompactness of an operator defined on a quasinormed space and taking values in a Marcinkiewicz space.
The main result is a general principle  from which it can be derived  that such operators are almost always maximally noncompact in the sense that their ball measure of noncompactness coincides with their operator norm.
We point out specifications of the universal principle to the case of the identity operator.
\looseness=-1
\end{abstract}

\author{Jan Mal\'y\textsuperscript{1}}
\address{\textsuperscript{1}%
Department of Mathematical Analysis,
Faculty of Mathematics and Physics,
Charles University,
So\-ko\-lo\-vsk\'a~83,
186~75 Praha~8,
Czech Republic}

\author{Zden\v{e}k Mihula\textsuperscript{2}}
\address{\textsuperscript{2}%
Department of Mathematics,
Faculty of Electrical Engineering,
Czech Technical University in Prague,
Technick\'a~2,
166~27 Praha~6,
Czech Republic}

\author{V\'\i t Musil\textsuperscript{3}}
\address{\textsuperscript{3}%
Masaryk University,
Faculty of Informatics,
Department of Computer Science,
Botanick\'a 68a,
602 00 Brno,
Czech Republic}

\author{Lubo\v s Pick\textsuperscript{1}}

\email[Z.\,Mihula, corresponding author]{mihulzde@fel.cvut.cz}

\urladdr[J.\,Mal\'{y}]{0000-0002-3615-9879}
\urladdr[Z.\,Mihula]{0000-0001-6962-7635}
\urladdr[V.\,Musil]{0000-0001-6083-227X}
\urladdr[L.\,Pick]{0000-0002-3584-1454}

\date{\today}


\thanks{This research was partly funded by Czech Science Foundation grants 18-00580S and 23-04720S. We thank the referee for their critical reading of the paper}

\maketitle

\bibliographystyle{abbrvnat}

\makeatletter
   \providecommand\@dotsep{2}
\makeatother
%\listoftodos\relax

\section{Introduction}

\subsection{Marcinkiewicz spaces}

Under abundant names and varying degrees of generality, numerous variants of \emph{Marcinkiewicz spaces} have been appearing in the literature for a long time, and their importance and arrays of applications are quite well known.

Let $(\RRR,\mu)$ denote a non-atomic $\sigma$-finite measure space of positive measure. Marcinkiewicz spaces $m_{\varphi}=m_{\varphi}(\RRR, \mu)$ and $M_{\varphi}=M_{\varphi}(\RRR, \mu)$ are defined as collections of $\mu$-measurable functions $f\colon\RRR\to\R$ such that $\nrm{f}_{m_{\varphi}}<\infty$ or $\nrm{f}_{M_{\varphi}}<\infty$, respectively,
where, loosely speaking, the functionals $\nrm{f}_{m_{\varphi}}$ and $\nrm{f}_{M_{\varphi}}$ control the size of $|f|$ using a fixed function $\varphi\colon(0,\mu(\RRR))\to(0,\infty)$.
Precisely,
\begin{equation*}
    \nrm{f}_{m_{\varphi}}
		= \sup_{t\in(0,\mu(\RRR))}\varphi(t)\,f^{*}(t)
    \quad\text{and}\quad
	\nrm{f}_{M_{\varphi}}
		= \sup_{t\in(0,\mu(\RRR))}\varphi(t)\,f^{**}(t),
\end{equation*}
where $f^*\colon(0,\infty)\to[0,\infty]$ is the \emph{nonincreasing rearrangement} of $f$, given by
\begin{equation*}
	f^*(t) = \inf\{\lambda>0:\mu(\{x\in\RRR: |f(x)|>\lambda\})\leq t\}
		\quad\text{for $t\in(0,\infty)$},
\end{equation*}
and $f^{**}\colon(0,\infty)\to[0,\infty]$ denotes the \emph{maximal nonincreasing rearrangement} of $f$ defined as
\begin{equation*}
	f^{**}(t) = \frac{1}{t}\int_0^tf^*(s)\,\d s
		\quad\text{for $t\in(0,\infty)$}.
\end{equation*}
The space $M_{\varphi}$ is always embedded into $m_{\varphi}$, thanks to the inequality $f^*\leq f^{**}$, and there are simple criteria on $\varphi$ under which they coincide.
Whereas the operation $f\mapsto f^{**}$ is sublinear, the operation $f\mapsto f^{*}$ is only positively homogeneous.
Consequently, $\nrm{\cdot}_{M_{\varphi}}$ is always a norm, whereas $\nrm{\cdot}_{m_{\varphi}}$ is typically merely a quasinorm.
This to some extent explains why the spaces $m_{\varphi}$ seem to appear less often than the spaces $M_{\varphi}$ or are completely omitted in the literature, \cf \eg monographs by~\citet{Kre:82} or~\citet{Ben:88}.

The spaces $M_{\varphi}$ first appeared in the work of~\citet{Lor:51}, where they were shown to be the dual spaces of certain separable function spaces, whose importance was spotted in connection with the limiting behaviour of potential operators, and later with fine properties of Sobolev functions, \cf~\eg~\citet{Ste:81}. Those function spaces were labelled as $\Lambda_{\varphi}$ spaces, and later called Lorentz spaces, or also Lorentz endpoint spaces.
A systematic and comprehensive treatment of both $\Lambda_{\varphi}$ and $M_{\varphi}$ spaces can be found in~\cite{Kre:82}, where $\varphi$ is required to be a nontrivial concave function.
Properties of these and related function spaces were studied, for example, by~\citet{Sem:63,Sem:64,Kre:72} and, under yet different assumptions on $\varphi$, by~\citet{Sha:72}.

Despite appearing less often, the spaces $m_{\varphi}$ should not be neglected either.
For example, the weak Lebesgue spaces $L^{p, \infty}$, which indisputably play an important role in harmonic analysis, are particular instances of $m_{\varphi}$, corresponding to the choice $\varphi(t) = t^{1/p}$, $p\in (0, \infty)$.
For $p\in(1,\infty)$, $M_{\varphi}$ with $\varphi(t) = t^{1/p}$ also coincides (up to equivalence of the defining functionals) with $L^{p, \infty}$, but this is not the case when $p\in(0, 1]$.
For $p = 1$, the Marcinkiewicz space $M_{\varphi}$ collapses to $L^1$ (with equal norms). When $p\in(0,1)$, $M_{\varphi}$ is $L^1$ again (with equivalent norms) provided that $\mu(R) < \infty$, otherwise it is trivial, \ie, it contains only the zero function.
Furthermore, the mere equivalence of the defining functionals is often simply not enough when dealing with problems of a geometric nature.
An interesting application of $m_{\varphi}$ spaces, which are equivalent neither to weak Lebesgue spaces nor to any $M_{\varphi}$ spaces, was recently found by~\citet[Theorem~3.1]{Cia:23}, where the choices of $\varphi(t) = t\log(e/t)$ and $\varphi(t) = t\sqrt{\log(e/t)}$ play a crucial role in investigation of sharp estimates of the Ornstein--Uhlenbeck operator on the Gaussian space.

\subsection{Noncompactness}

Our main focus is on (non)compactness of operators into Marcinkiewicz spaces and the quantitative analysis of how much noncompact particular operators having values in Marcinkiewicz space are.
The quantitative approach requires finer means than those used for mere analysis of whether an operator is compact or not.

Our approach to measuring the lack of compactness is based on the so-called ball measure of noncompactness.
Let $X$ and $Y$ be (quasi)normed linear spaces and let $T$ be a bounded positively homogeneous mapping defined on $X$ and taking values in $Y$, a fact we will denote by $T\colon X\to Y$.
The \emph{ball measure of noncompactness} $\bmc(T)$ of $T$ is defined as the infimum of radii $r>0$ for which there exists a finite collection $\{g_j\}_j$ of elements in $Y$ such that
\begin{equation*}
    T(B_X)\subseteq\bigcup_j \left(g_j + r B_Y \right),
\end{equation*}
where $B_X$ and $B_Y$ denote the closed unit balls in $X$ and $Y$, respectively.
Our aim is to provide general theorems suitable for obtaining lower bounds on the ball measure of noncompactness of operators having Marcinkiewicz spaces as their targets.
Loosely speaking, such lower bounds tell us ``how bad at least'' the noncompactness has to be.

The ball measure of noncompactness was introduced by~\citet{Kur:30}, and it has been heavily studied and applied ever since, see, \eg \citet{Dar:55,Sad:68,Ban:80,Car:80,Car:81,Hen:03,Edm:18,Bou:19,Lan:19,Lan:20} and references therein.

Clearly, one always has $0 \le \bmc(T) \le \nrm{T}$, where $\nrm{T}$ denotes the operator norm, and it can be easily observed that $T$ is compact if and only if $\bmc(T)=0$.
An operator $T$ will be called \emph{maximally noncompact} if $\alpha(T)=\|T\|$.

\subsection{Sobolev embeddings}

Embeddings between function spaces are probably the most important and arguably also most frequently appearing in the literature category of operators whose compactness and noncompactness issues are of interest.
An \emph{embedding} is the action of the identity operator acting between two function spaces.
A crucial example is \emph{Sobolev embeddings}, where, surprising as it may seem, Marcinkiewicz spaces appear naturally.

A leading example is closely connected with the \emph{limiting} cases of Sobolev embeddings, that is, those of the homogeneous Sobolev space $V_{0}^{1,n}(\Omega)$ (and its various modifications), of functions defined in a bounded domain $\Omega$ in $\mathbb R^n$, $n\geq2$ and vanishing at the boundary of $\Omega$ in a suitable sense.
The adjective ``limiting'' stems from the fact that the degree of integrability of derivatives, $n$, coincides with the dimension of the underlying domain.
Perhaps the most classical form of such a Sobolev embedding reads
\begin{equation}\label{E:trudinger}
    V_{0}^{1,n}(\Omega) \to \exp L^{n'}(\Omega),
\end{equation}
in which $n'=n/(n-1)$.
Embedding~\eqref{E:trudinger} was explicitly stated by~\citet{Tru:67}, but it can be derived from various other sources, see \eg~\citep{Yud:61,Poh:65,Hun:66,Pee:66,One:68,Cwi:00,Maz:11}.
The target space $\exp L^{n'}(\Omega)$ is usually interpreted as an Orlicz space.
Nevertheless, setting
\begin{equation}\label{E:trudinger_marc_phi}
    \varphi(t)=\left(\log \frac{2|\Omega|}{t}\right)^{-\frac{1}{n'}}
        \quad\text{for $t\in(0, |\Omega|)$},
\end{equation}
one can observe that $\exp L^{n'}(\Omega) = m_{\varphi}(\Omega) = M_{\varphi}(\Omega)$, in the sense that the function spaces coincide and, at the same time, their (quasi-)norms are equivalent, see \eg~\citep{OP:99}. 
However, one should be aware of an important catch here:
While the equivalence of the (quasi-)norms guarantees that the topological properties of all three spaces are the same, this does not mean that so are their geometrical properties, such as the measure of noncompactness of operators into them.
In particular, even though it was proved by \citet{Hen:03} that limiting embedding~\eqref{E:trudinger} is maximally noncompact when treating~$\exp L^{n'}$ as an Orlicz space, the same conclusion when the target space is replaced by the corresponding Marcinkiewicz space $m_{\varphi}$ or~$M_{\varphi}$ does not follow from this.
In fact, this was proved much later by~\citet{LMP:preprint}.

Another important connection to Sobolev embeddings is more subtle, and it is connected with optimal fundamental functions of target spaces.
The fundamental function of a rearrangement-invariant (quasi)normed function space $X$ is defined as $t\mapsto\nrm{\chi_E}_X$ where $E$ is any measurable set and $\mu(E)=t$.
Marcinkiewicz spaces are closely tied to the study of fundamental functions.
In particular, the space $M_{\varphi}$ is known to be the largest possible rearrangement-invariant Banach function space with a prescribed fundamental function, see~\citet{Ben:88}.
It has been shown that a Sobolev embedding is noncompact even after its optimal target is enlarged to the Marcinkiewicz space having the same fundamental function, \cf~\citep{Ker:08,Sla:12,Sla:15,Fer:10,Cav:19}.
Therefore, Marcinkiewicz spaces mark an important threshold.
However, it should be noted that the said threshold between compactness and noncompactness is rather blurred, as demonstrated recently by~\citet{Lan:22} who proved that, for a considerably general class of Sobolev embeddings, the Marcinkiewicz target space can actually still be essentially enlarged, and yet the resulting embedding remains noncompact.

\subsection{Disjoint superadditivity}

The disjoint superadditivity constitutes an important geometric property in connection with noncompactness.
A (quasi)normed linear space of functions $X$ is said to be \emph{disjointly superadditive} if there
exist $\gamma>0$ and $C>0$ such that for every $m\in\N$ and
every collection of functions $\{f_k\}_{k=1}^m\subseteq X$ with pairwise disjoint supports, one has
\begin{equation}\label{intro:disjoint_superadditivity}
	\sum_{k=1}^{m} \nrm{f_k}_{X}^\gamma
		\le C \nrm*{\sum_{k=1}^{m} f_k}_{X}^\gamma.
\end{equation}
This property was intensively exploited by~\citet{Hen:03} to prove that the nonlimiting Sobolev embedding $V_{0}^{1,p}(\Omega) \to L^{p^*}(\Omega)$, where $p\in[1, n)$ and $p^* = np/(n - p)$, is maximally noncompact.
Later, it was also heavily used by~\citet{Bou:19} to prove that nonlimiting Sobolev embeddings into two-parameter Lorentz spaces $L^{p^*,q}(\Omega)$, where $q\in[p, \infty)$, are also maximally noncompact.
The advances made in those papers indicate that this approach could be most likely extended from Sobolev embeddings to general operators, should the target space be disjointly superadditive.
Notably, the case $q = \infty$, where the target space is a Marcinkiewicz space, was completely avoided there.
The missing gap was later filled by~\cite{Lan:20}, where it was also shown that weak Lebesgue spaces are not disjointly superadditive.
This raises the question whether something can be said about disjoint superadditivity of Marcinkiewicz spaces in general.
It should be mentioned that, in the theory of Banach spaces/lattices and their geometry, the disjoint superadditivity is often called a lower $\gamma$-estimate, see \eg~\citep{LT:79}.
Nevertheless, we prefer to stick with the terminology used in~\citep{Hen:03} because we think that it captures the essence of how this property is often used for proving maximal noncompactness.

\subsection{Overview of the main results}

Although disjoint superadditivity proved to be a useful tool when showing maximal noncompactness, it is usually not at our disposal when dealing with Marcinkiewicz spaces.
In Section~\ref{S:disjoint-superadditivity}, we elaborate on this.
Theorem~\ref{T:M-not-superadditive} shows that \textbf{Marcinkiewicz spaces $M_{\varphi}$ are almost never disjointly superadditive}, with the exception of when they collapse to $L^1$.
Next, Theorem~\ref{T:m-not-superadditive} gives \textbf{a sufficient condition for the lack of disjoint superadditivity of $m_{\varphi}$}.
Note that the given condition is rarely violated in typical situations.

Section~\ref{S:maximal-noncompactness} contains our central results.
The abstract results of Theorems~\ref{T:general-lower} and \ref{T:maximal_noncomp_alt_m_phi} provide us with \textbf{a general lower bound for the measure of noncompactness of a bounded positively homogeneous operator}.
Here, we focus on identifying minimal assumptions that are needed in later applications.
Next, we apply the result to obtain \textbf{a lower bound for the measure of noncompactness of an embedding operator} and, consequently, \textbf{a sufficient condition for its maximal noncompactness}, see Theorem~\ref{T:max-noncompact} and Corollary~\ref{C:max-noncompact_shrinking}.
We emphasize that Theorem~\ref{T:general-lower} is not a generalization of the method by~\citet{Lan:20}. The method there exploits the absolute continuity of target (quasi)norms, which is not at our disposal when dealing with Marcinkiewicz spaces (save for the case $L^1$). In particular, that method is not suitable for proving the maximal noncompactness of the limiting embedding $V_{0}^{1,n}(\Omega) \to m_\varphi(\Omega)$ with $\varphi$ as in~\eqref{E:trudinger_marc_phi}.

In Section~\ref{S:embeddings-into-l-infty}, we address the problem of maximal noncompactness of embeddings into $L^{\infty}$.
Although~$L^{\infty}$ is technically also a Marcinkiewicz space of both types, the situation significantly differs from the other cases.
As shown by~\citet{Bou:19} and \citet{Lan:20}, embeddings into $L^{\infty}$ are often \emph{not} maximally noncompact.
An upper bound on the measure of noncompactness can be typically obtained by exploiting the span, see~\cite[Proposition~5.1]{Lan:20}.
To show that the upper bound matches the measure of noncompactness, one usually proves a lower bound, which requires more subtle techniques. 
In Theorem~\ref{T:beta-lower-of-X-to-ell-infty}, we provide a new proof of a very \textbf{general lower bound for the measure of noncompactness of an embedding from a general quasinormed space into $L^{\infty}$}.
In particular, we avoid the combinatorial argument given by~\citet{Lan:20}.

\section{Lack of disjoint superadditivity of Marcinkiewicz spaces}
\label{S:disjoint-superadditivity}

In this section we prove that the Marcinkiewicz spaces are usually not disjointly superadditive. Loosely speaking, $m_\varphi(\RRR,\mu)$ is not disjointly superadditive whenever it is a nontrivial quasi-Banach space, whereas $M_\varphi(\RRR,\mu)$ is not disjointly superadditive whenever it is a nontrivial Banach space not equivalent to $L^1(\RRR,\mu)$.

We start by recalling some terminology and basic properties of Marcinkiewicz spaces.

First, for all measurable functions $f,g\colon\RRR\to\R$, one has
\begin{equation} \label{E:f**-subadditivity}
	(f+g)^{**}(t) \le f^{**}(t) + g^{**}(t)
		\quad\text{for $t\in(0,\infty)$}
\end{equation}
and
\begin{equation} \label{E:f*-subadditivity}
	(f+g)^{*}(s+t) \le f^{*}(s) + g^{*}(t)
		\quad\text{for $s,t\in(0,\infty)$}.
\end{equation}

We recall that a functional $\nrm{\cdot}_X\colon X \to [0, \infty)$, where $X$ is a vector space, is called a \emph{quasinorm} if it satisfies the following conditions:
\begin{itemize}
    \item it is positively homogeneous, that is, for every $\lambda\in\R$ and $f\in X$ one has $\nrm{\lambda f}_X=|\lambda|\nrm{f}_X$,
    \item it satisfies $\nrm{f}_X=0$ if and only if $f = 0$ in $X$,
    \item there is a constant $C \ge 1$ such that for every $f,g\in X$ one has $\nrm{f + g}_X\le C(\nrm{f}_X+\nrm{g}_X)$.
\end{itemize}
We denote the least such $C$ by $C_X$. Clearly (unless $X$ is trivial), $C_X\in[1, \infty)$, and $\|\cdot\|_{X}$ is a norm if and only if $C_X = 1$. For more details concerning quasinormed Banach function spaces see~\cite{Nek:24}.

We say that $\varphi\colon(0,\mu(\RRR))\to(0,\infty)$ is \emph{quasiconcave} if it is nondecreasing and the function $\varphi(t)/t$ is nonincreasing on $(0,\mu(\RRR))$. Quasiconcave functions are always continuous on $(0,\mu(\RRR))$. Furthermore, they are often in fact assumed to be defined on the interval $[0,\mu(\RRR))$ and are required to satisfy $\varphi(t) = 0$ if and only if $t = 0$. For example, the function $\varphi(t) = t^\alpha$, $t\in(0, \mu(\RRR))$, is quasiconcave if and only if $\alpha\in[0,1]$.

We say that a function $\varphi\colon(0,\mu(\RRR))\to(0,\infty)$ is \emph{admissible} if the function $\tff$ defined as
\begin{equation*}
   \tff(t) =  \sup_{s\in(0,t]} s \sup_{\tau\in[s, \mu(\RRR))} \frac{\varphi(\tau)}{\tau} = t \sup_{s\in[t, \mu(\RRR))} \frac{\sup_{\tau\in(0, s]} \varphi(\tau)}{s},\ t\in(0, \mu(\RRR)),
\end{equation*}
is finite. Given an admissible function $\varphi$, the function $\tff$ is called the \emph{least quasiconcave majorant} of $\varphi$. The quasiconcave majorant is quasiconcave and satisfies $\varphi\leq\tff$. Moreover, we have (see~\cite[Lemma~1.5]{GP:06})
\begin{equation}\label{E:quasiconcave_majorant_equal_norms_M}
    \|f\|_{M_{\varphi}} = \|f\|_{M_{\tff}} \quad \text{for every measurable function $f$}.
\end{equation}
Obviously, when $\varphi$ is quasiconcave, then it is admissible and $\varphi = \tff$. Assuming that $\mu(\RRR)<\infty$, a textbook example of an admissible function is the function $\varphi(t) = t^\alpha \log(2\mu(\RRR)/t)^\beta$, $t\in(0, \mu(\RRR))$, where either $\alpha > 0$ and $\beta\in\R$ or $\alpha = 0$ and $\beta\leq0$. The Marcinkiewicz space $M_{\varphi}$ is a nontrivial Banach space provided that $\varphi$ is an admissible function. In particular, the fact that $\nrm{\,\cdot\,}_{M_\varphi}$ satisfies the triangle inequality follows from \eqref{E:f**-subadditivity}.

We say that a function $\varphi\colon(0,\mu(\RRR))\to(0,\infty)$ satisfies the \textit{$\Delta_2$
condition}, and write $\varphi\in\Delta_2$, if there is $c>0$ such that
\begin{equation*}
	\varphi(2t) \le c\varphi(t)
		\quad\text{for every $t\in(0,\mu(\RRR)/2)$.}
\end{equation*}
We say that $\varphi$ satisfies the $\Delta_2$ condition \emph{near zero} if $\varphi(2t) \le c\varphi(t)$ for every $t\in(0,t_0)$ for some $t_0\in(0,\mu(\RRR)/2)$. The space $m_\varphi$ is a nontrivial quasi-Banach space provided that $\varphi\in\Delta_2$. In such case, the quasinorm constant of $\nrm{\cdot}_{m_{\varphi}}$ is not larger than the $\Delta_2$-constant of $\varphi$. Indeed, by \eqref{E:f*-subadditivity}, we have
\begin{equation*}
	\nrm{f+g}_{m_\varphi}
		= \sup_{t\in(0,\mu(\RRR))} \varphi(t)(f+g)^*(t)
		\le \sup_{t\in(0,\mu(\RRR)/2)} \varphi(2t)\bigl(f^*(t) + g^*(t)\bigr)
		\le c\bigl( \nrm{f}_{m_\varphi} + \nrm{g}_{m_\varphi} \bigr),
\end{equation*}
for measurable $f,g\colon\RRR\to\R$, where $c$ is the $\Delta_2$-constant of $\varphi$. If $\varphi$ is quasiconcave, then $\varphi\in\Delta_2$ with $c=2$. Indeed, since the function $t\mapsto\varphi(t)/t$ is nonincreasing, we have
\begin{equation} \label{E:delta2}
	\frac{\varphi(2t)}{\varphi(t)}
		= 2 \frac{\varphi(2t)}{2t}\frac{t}{\varphi(t)}
		\le 2 \quad\text{for every $t\in(0,\mu(\RRR)/2)$}.
\end{equation}

We say that $\varphi$ is \emph{almost quasiconcave} if it is admissible and there is a constant $C_\varphi\in(0, 1]$ such that
\begin{equation}\label{E:quasiconcave_majorant_equivalence}
    C_\varphi \tff \leq \varphi \leq \tff \quad \text{in $(0, \mu(\RRR))$}.
\end{equation}
Clearly, if $\varphi$ is quasiconcave, then it is also almost quasiconcave with $C_\varphi = 1$. Unlike in \eqref{E:quasiconcave_majorant_equal_norms_M}, the functionals $\| \cdot \|_{m_{\varphi}}$ and $\| \cdot \|_{m_{\tff}}$ do not coincide in general, but they are equivalent. More precisely, we have
\begin{equation}\label{E:quasiconcave_majorant_equivalent_norms_m}
   C_\varphi \|f\|_{m_{\tff}} \leq \|f\|_{m_{\varphi}} \leq \|f\|_{m_{\tff}} \quad \text{for every measurable function $f$}.
\end{equation}
Moreover, if $\varphi$ is almost quasiconcave, then $\varphi\in\Delta_2$ with $c = 2/C_\varphi$. In particular, $m_\varphi$ is a nontrivial quasi-Banach space when $\varphi$ is almost quasiconcave. In the following two examples, we assume $\mu(\RRR) < \infty$. A typical example of an almost quasiconcave function is the function $\varphi(t) = t^\alpha \log(2\mu(\RRR)/t)^\beta$, $t\in(0, \mu(\RRR))$, where $\alpha\in(0,1)$ and $\beta\in\R$, $\alpha = 0$ and $\beta\leq0$, or $\alpha = 1$ and $\beta\geq0$. A textbook example of a function that is admissible but not almost quasiconcave is the function $\varphi(t) = t^\alpha$, $t\in(0, \mu(\RRR))$, where $\alpha > 1$.

Finally, we are in a position to state and prove the statements of this section. We shall treat each of the two types of Marcinkiewicz spaces separately. The results are known in the case when $\varphi$ is a power function, see~\cite[Theorems~2.1~and~2.2]{Lan:20}.
Note that \citet{KK:05} studied when $m_\varphi(\RRR, \mu)$ satisfies the so-called upper $r$-estimate, which is the reverse inequality of \eqref{intro:disjoint_superadditivity} with $\gamma=r$. Moreover, their assumptions on $\varphi$ are slightly more restrictive and, more importantly, the space $M_\varphi(\RRR, \mu)$ is not considered.

\begin{theorem} \label{T:m-not-superadditive}
Let $\varphi\colon(0,\mu(\RRR))\to(0,\infty)$ be a nondecreasing function
satisfying $\Delta_2$ condition near zero. Then the space $m_\varphi(\RRR, \mu)$ is not disjointly superadditive.
\end{theorem}

\begin{proof}
By assumption, there exist $t_0>0$ and $c>0$
such that $\varphi(2t)\le c\varphi(t)$ for $t\in(0,t_0)$.
Fix $m\in\N$. There are pairwise disjoint
measurable subsets $E_k$, $k=1,\dots,m$, of $\RRR$ such that their measures
$r_k=\mu(E_k)$ satisfy
\begin{equation} \label{E:rk}
	r_{k+1}\le \frac{r_k}{2}
		\quad\text{for each $k=1,\dots,m-1$}.
\end{equation}
Moreover, we may assume that $r_1\in(0,t_0)$.
We define the functions
\begin{equation} \label{E:fk-def}
	f_k=\frac{\chi_{E_k}}{\varphi(r_k)}
		\quad\text{for $k=1,\dots,m$}.
\end{equation}
Then
\begin{equation} \label{E:fk-rearrangement}
	f_k^*=\frac{\chi_{(0,r_k)}}{\varphi(r_k)},
\end{equation}
and thus
\begin{equation}\label{E:fk-small_m_norm}
	\nrm{f_k}_{m_\varphi}
		= \sup_{t\in(0,\mu(\RRR))} \varphi(t)\, f_k^*(t)
		= 1\quad\text{for $k=1,\dots,m$}.
\end{equation}
We set
\begin{equation} \label{E:f-def}
	f = \sum_{k=1}^m f_k.
\end{equation}
Since the functions $f_k$, $k = 1, \dots, m$, have pairwise disjoint supports, we have
\begin{equation} \label{E:fks-rearrangement}
	f^*
		= \sum_{k=1}^{m} \frac{\chi_{(a_k,a_{k-1})}}{\varphi(r_k)},
\end{equation}
where
\begin{equation*}
	a_k=
	\begin{cases}
		r_{k+1}+\dots+r_m &\text{if $k=0,\dots,m-1$},
			\\
		0 &\text{if $k=m$}.
	\end{cases}
\end{equation*}
Consequently,
\begin{align}
	\nrm{f}_{m_\varphi}
		& = \sup_{t\in(0,\mu(\RRR))} \varphi(t)\, f^*(t)
			= \max_{j\in\{1,\dots,m\}} \sup_{t\in(a_j,a_{j-1})}
            \varphi(t)
			\sum_{k=1}^m \frac{1}{\varphi(r_k)} \chi_{(a_k,a_{k-1})}(t)
			\nonumber \\
		& = \max_{j\in\{1,\dots,m\}} \frac{\varphi(a_{j-1})}{\varphi(r_j)}
			\le \max_{j\in\{1,\dots,m\}} \frac{\varphi(2r_j)}{\varphi(r_j)}
			\le c, \label{E:f-small_m_norm}
\end{align}
where we used \eqref{E:rk} to show that
\begin{equation} \label{E:aj2rj}
	a_{j-1} = r_j+\cdots+r_m \le 2r_j
		\quad\text{for $j=1,\dots,m$}
\end{equation}
and the $\Delta_2$ condition.

Now, suppose that $m_{\varphi}(\RRR, \mu)$ is disjointly superadditive. Then there are $\gamma>0$ and a positive constant
$C$ such that
\begin{equation}\label{E:ds}
	\sum_{k=1}^{m} \nrm{f_k}_{m_{\varphi}}^\gamma
		\le C \nrm{f}_{m_{\varphi}}^\gamma.
\end{equation}
However, thanks to \eqref{E:fk-small_m_norm} and \eqref{E:f-small_m_norm}, this implies $m \le Cc^\gamma$, which is clearly impossible because $m\in\N$ was selected
arbitrarily at the beginning.
\end{proof}

\begin{remark}
The conclusion of Theorem~\ref{T:m-not-superadditive} is still true even when $\varphi$ is not necessarily nondecreasing but merely equivalent to a nondecreasing function. By that, we mean that there are constants $C_1,C_2>0$ and a nondecreasing function $\psi\colon (0, \mu(\RRR)) \to (0, \infty)$ such that $C_1\psi \leq \varphi \leq C_2\psi$ on $(0, \mu(\RRR))$.
\end{remark}

We now turn our attention to the other type of Marcinkiewicz spaces.

\begin{theorem} \label{T:M-not-superadditive}
Let $\varphi\colon (0, \mu(\RRR)) \to (0, \infty)$ be a quasiconcave function.  Then
the space $M_\varphi(R, \mu)$ is not disjointly superadditive if and only if
\begin{equation}\label{E:varphi-lim}
	\lim_{t\to0_+}\frac{t}{\varphi(t)} = 0.
\end{equation}
\end{theorem}

\begin{proof}
Assume that \eqref{E:varphi-lim} is true. Fix $m\in\N$ and let the sets $\{E_k\}_{k=1}^m$ be chosen as in
the proof of Theorem~\ref{T:m-not-superadditive}. Moreover, we may assume in addition that 
\begin{equation}
\label{E:rk-phi}
	\frac{r_{k+1}}{\varphi(r_{k+1})}
		\le \frac{r_k}{2\varphi(r_k)}
		\quad\text{for each $k=1,\dots,m-1$},
\end{equation}
thanks to \eqref{E:varphi-lim}.
Furthermore, let the functions $f_k$, $k=1,\dots,m$, be defined as in \eqref{E:fk-def}, and we denote their sum by $f$ as in \eqref{E:f-def}.
Note that
\begin{equation*}
	f_k^{**}(t)
		= \frac{1}{\varphi(r_k)} \chi_{(0,r_k)}(t)
			+ \frac{r_k}{t\varphi(r_k)} \chi_{[r_k,\infty)}(t)
                \quad\text{for every $t\in(0,\infty)$}
\end{equation*}
by \eqref{E:fk-rearrangement}, and so, using the quasiconcavity of $\varphi$,
\begin{align*}
	\nrm{f_k}_{M_{\varphi}}
		& = \sup_{t\in(0,\mu(\RRR))} \varphi(t)\,f_k^{**}(t)
		= \max\left\{\frac{1}{\varphi(r_k)}\sup_{t\in(0,r_k)}\varphi(t),
				\frac{r_k}{\varphi(r_k)}\sup_{t\in[r_k,\mu(\RRR))}\frac{\varphi(t)}{t}\right\}
		= 1
\end{align*}
for every $k=1,\dots,m$.
Now, fix $j\in\{1,\dots,m-1\}$ and $t\in(a_j,a_{j-1}]$. Then,
using \eqref{E:fks-rearrangement},
\begin{equation*}
	\int_0^t f^*(s)\,\d s
		= \int_{0}^{a_j} f^*(s)\,\d s
			+ \int_{a_j}^{t} f^*(s)\,\d s
		= \sum_{k=j+1}^m\frac{r_k}{\varphi(r_k)} + \frac{t-a_j}{\varphi(r_j)}.
\end{equation*}
Owing to the quasiconcavity of $\varphi$, we conclude that
\begin{align*}
	\begin{split}
	\sup_{t\in(a_j,a_{j-1}]} \varphi(t)\, f^{**}(t)
		& \le \sup_{t\in(a_j,a_{j-1}]}
			\left(\frac{\varphi(t)}{t} \sum_{k=j+1}^{m}\frac{r_k}{\varphi(r_k)}
				+\frac{\varphi(t)}{\varphi(r_j)}\right)
			\\
		& \le \frac{\varphi(a_j)}{a_j} \sum_{k=j+1}^{m}\frac{r_k}{\varphi(r_k)}
			+\frac{\varphi(a_{j-1})}{\varphi(r_j)}.
	\end{split}
\end{align*}
Note that $a_{j-1}\le 2r_j$ by \eqref{E:aj2rj}, and
\begin{equation*}
	\frac{\varphi(a_{j-1})}{\varphi(r_j)}
		\le \frac{\varphi(2r_j)}{\varphi(r_j)}
		\le 2
\end{equation*}
thanks to \eqref{E:delta2}. Hence
\begin{equation}\label{E:sup-ajs}
\sup_{t\in(a_j,a_{j-1}]} \varphi(t)\, f^{**}(t) \leq 2 + \frac{\varphi(a_j)}{a_j} \sum_{k=j+1}^{m}\frac{r_k}{\varphi(r_k)}.
\end{equation}
As for the sum on the right-hand side, we obtain
\begin{equation*}
	\sum_{k=j+1}^{m} \frac{r_k}{\varphi(r_k)}
		\le \frac{r_{j+1}}{\varphi(r_{j+1})} \sum_{k=j+1}^{m} 2^{j+1-k}
		\le 2 \frac{r_{j+1}}{\varphi(r_{j+1})}
\end{equation*}
by \eqref{E:rk-phi}. Furthermore, since $a_j\ge r_{j+1}$ by the definition of $a_j$, we have
\begin{equation*}
    \frac{\varphi(a_j)}{a_j}\le \frac{\varphi(r_{j+1})}{r_{j+1}}.
\end{equation*}
Hence
\begin{equation*}
\sum_{k=j+1}^{m} \frac{r_k}{\varphi(r_k)} \leq 2 \frac{a_j}{\varphi(a_j)}.
\end{equation*}
Combining this with \eqref{E:sup-ajs}, we arrive at
\begin{equation} \label{E:sup1}
	\sup_{t\in(a_j,a_{j-1}]} \varphi(t)\, f^{**}(t)
		\le 4
		\quad\text{for each $j=1,\dots, m-1$}.
\end{equation}

It remains to consider the case when $t\in(0,a_{m-1}]$. But, for such $t$,
\begin{equation*}
	\int_0^t f^*(s)\,\d s
		= \frac{t}{\varphi(r_m)},
\end{equation*}
and so
\begin{equation} \label{E:sup2}
	\sup_{t\in(0,a_{m-1}]} \varphi(t)f^{**}(t)
		= \sup_{t\in(0,a_{m-1}]} \frac{\varphi(t)}{\varphi(r_m)}
		= \frac{\varphi(a_{m-1})}{\varphi(r_m)}
		= 1.
\end{equation}
Now, combining the estimates \eqref{E:sup1} and \eqref{E:sup2}, we obtain
$\nrm{f}_{M_\varphi}\le 4$. Therefore, if $M_{\varphi}(\RRR, \mu)$ were disjointly superadditive, then, as in~\eqref{E:ds}, we would obtain $m\le C4^\gamma$ for some $\gamma>0$ and $C>0$. However, this is impossible because $m\in\N$ was arbitrary.

Finally, assume that~\eqref{E:varphi-lim} is not satisfied. Since the function $t\mapsto \frac{t}{\varphi(t)}$ is nondecreasing on $(0,\mu(\RRR))$, the limit in \eqref{E:varphi-lim} exists, and it is finite and positive. In other words, we have
\begin{equation*}
	\lim_{t\to 0_+}\frac{t}{\varphi(t)} \in(0, \infty).
\end{equation*}
It is easy to see that this implies that $M_{\varphi}(\RRR, \mu)$ is equivalent to $L^{1}(\RRR, \mu)$ in the sense that the spaces are equal as sets and their norms are equivalent. Since $L^{1}(\RRR, \mu)$ is obviously disjointly superadditive, so is $M_{\varphi}(\RRR, \mu)$.
\end{proof}

\begin{remark}
The conclusion of Theorem~\ref{T:M-not-superadditive} is still true if $\varphi$ is merely an almost quasiconcave function.
\end{remark}

\section{Maximal noncompactness of embeddings into a Marcinkiewicz space}\label{S:maximal-noncompactness}

In this section we present the main results of the paper. Loosely speaking, they state that embeddings into Marcinkiewicz spaces with a certain shrinking property are maximally noncompact even though they are hardly ever disjointly superadditive (as we already know). We develop a new approach to the problem, whose core is the following abstract result. Its application will follow immediately. It will be useful to note that, for a $\mu$-measurable function $f\colon\RRR\to\R$, one has
\begin{equation}\label{E:f*-alt}
    f^*(t_-) = \sup_{\mu(E)=t}\essinf_{x\in E}|f(x)| \quad \text{for $t\in(0,\mu(\RRR))$}
\end{equation}
and
\begin{equation}\label{E:f**-alt}
    f^{**}(t) = \frac{1}{t}\sup_{\mu(E)=t}\int_E|f(x)|\d\mu(x)  \quad \text{for $t\in(0,\mu(\RRR))$}.
\end{equation}

We say that an operator $T\colon X \to Y$ is \emph{positively homogeneous} if $\|T(\gamma x)\|_Y = \gamma \|Tx\|_Y$ for every $x\in X$ and every $\gamma > 0$.

We first state an elementary but useful observation that will prove to be quite useful in proofs of the main results.

\begin{lemma}\label{L:distance}
    Let $X,Y$ be quasinormed spaces,  $T\colon X\to Y$ a bounded positively homogeneous operator, and $r>0$. If $g\in Y$ is such that
    \begin{equation*}
        T(B_X)\cap\left(g+rB_Y\right)\neq\emptyset,
    \end{equation*}
    then
    \begin{equation}\label{E:distance}
        \|g\|_{Y}\le C_Y\left(\|T\|+r\right),
    \end{equation}
    in which $C_Y$ is the constant in the definition of a quasinorm.
\end{lemma}

\begin{proof}
  Assume that~\eqref{E:distance} is not satisfied, that is, 
  \begin{equation*}
      \|g\|_{Y}> C_Y\left(\|T\|+r\right).
  \end{equation*}
  Then, for every $f\in B_X$, one has
  \begin{equation*}
      \|Tf\|_Y \le \|T\|\,\|f\|_X \le\|T\|,
  \end{equation*}
  whence
  \begin{equation*}
      \|g-Tf\|_Y \ge C_Y^{-1}\|g\|_Y-\|Tf\|_Y
      > C_Y^{-1}\left(C_Y\left(\|T\|+r\right)\right)-\|T\|=r.
  \end{equation*}
  Consequently, 
    \begin{equation*}
        T(B_X)\cap\left(g+rB_Y\right)=\emptyset.\qedhere
    \end{equation*}
\end{proof}

We are now in position to state and prove our first principal result.

\begin{theorem}\label{T:general-lower}
Let $X$ be a quasinormed space of measurable functions on $(\RRR, \mu)$. Let $\varphi\colon(0, \mu(\RRR)) \to (0, \infty)$. Assume one of the following:
\begin{enumerate}[label={\rm(\alph*)},ref={\rm(\alph*)}]
    \item\label{item:small_m}  $Y = m_\varphi(\RRR, \mu)$, $\varphi$ is almost quasiconcave, and there exists $t_0\in(0, \mu(\RRR))$ such that $\frac{\varphi(t)}{t}$ is nonincreasing on the interval $(0, t_0)$;
    \item\label{item:capital_M} $Y = M_\varphi(\RRR, \mu)$ and $\varphi$ is admissible.
\end{enumerate}
Let $T\colon X \to Y$ be a bounded positively homogeneous operator. Let $r\in(0, \|T\|)$.

Assume that there are $\tau\in(0,1]$ and a set $S\subseteq\RR$ of finite positive measure such that for each $\sigma >0$ and $\varepsilon\in(0, 1)$ there are $m\in\en$, functions $f_i\in X$ and pairwise disjoint sets $E_i\subseteq S$, $i=1,\dots,m$, each of positive measure, which in the case \ref{item:small_m} is smaller than $t_0$, and such that the following properties hold:
\begin{enumerate}[label={\rm(\roman*)},ref={\rm(\roman*)}]
\item\label{item:unit_ball} $\|f_i\|_X=1$ for every $i\in\{1,\dots,m\}$,
\item\label{item:large_measure} $\sum_{i=1}^m \mu(E_i)\ge \tau \mu(S)$,
\item\label{item:uniform_bound_from_below} $s_i\ge \sigma$ for every $i\in\{1,\dots,m\}$,
\item\label{item:large_marc_norm} $s_i\varphi(\mu(E_i))\ge (1-\varepsilon)r$ for every $i\in\{1,\dots,m\}$,
\end{enumerate}
where
\begin{equation*}
    s_i = \begin{cases}
        \essinf_{E_i} |Tf_i| \quad &\text{if $Y=m_{\varphi}(\RRR, \mu)$},\\
        \frac1{\mu(E_i)}\int_{E_i} |Tf_i| \, \d\mu \quad &\text{if $Y=M_{\varphi}(\RRR, \mu)$}.
    \end{cases}
\end{equation*}
Then $\bmc(T)\ge r$.
\end{theorem}

\begin{proof}
Let $\tilde{\varphi}$ be the least quasiconcave majorant of $\varphi$. In the case \ref{item:small_m}, let $C_\varphi\in(0, 1]$ be the constant from \eqref{E:quasiconcave_majorant_equivalent_norms_m}.

Suppose that $\bmc(T)<r$. Choose $\varepsilon\in(0, 1)$ so close to $0$ that
\begin{equation*}
    \bmc(T) < (1 - \varepsilon)^3 r < r.
\end{equation*}
Since $\bmc(T) < (1 - \varepsilon)^3 r$, there are $k\in\N$ and functions $g_j$, $j\in\{1,\dots,k\}$, such that
\begin{equation}\label{E:union-new}
    T(B_X)\subseteq \bigcup_{j=1}^k (g_j + (1-\varepsilon)^3rB_Y).
\end{equation}
Note that we may assume that
\begin{equation}\label{EE:out}
	\nrm{g_j}_{Y}\le C(\nrm{T}+\rpar)
		\quad\text{for each $j\in\{1,\dots,k\}$},
\end{equation}
where
\begin{equation*}
    C = \begin{cases}
        \frac{2}{C_\varphi} \quad &\text{if $Y=m_\varphi$,} \\
        1 \quad &\text{if $Y=M_\varphi$.}
    \end{cases}
\end{equation*}
Indeed, if, for some $j\in\{1,\dots,k\}$, one would have $\nrm{g_j}_{Y}>C(\nrm{T}+\rpar)$, then, by Lemma~\ref{L:distance}, one would get
\begin{equation*}
	T(B_{X})
		\cap \left(g_j + (1-\varepsilon)^3rB_{Y}\right)
		= \emptyset,
\end{equation*}
and so the function $g_j$ can be excluded from the collection on the right-hand
side of~\eqref{E:union-new}.

Now, choose $\sigma>0$ so large that either
\begin{equation}\label{E:P:general-lower:choice_sigma_small_m}
\varepsilon \sigma \tff\big(\varepsilon\tfrac{\tau}{k}\;\mu(S)\big) > \frac{2}{C_\varphi^2}(\nrm{T}+\rpar)
\end{equation}
if $Y=m_\varphi$ or
\begin{equation}\label{E:P:general-lower:choice_sigma_capital_M}
\varepsilon \sigma \tff(\tfrac{\tau}k\mu(S)) > \nrm{T}+\rpar
\end{equation}
if $Y=M_\varphi$. We can now find $m$, $f_i$ and $E_i$ for $i\in\{1,\dots,m\}$, whose existence is guaranteed by the assumptions.
We claim that there is $j\in\{1,\dots,k\}$ such that
\begin{equation}\label{E:new-star}
    \sum_{i\in A_j} \mu(E_i) \ge \frac{\tau}{k}\;\mu(S),
\end{equation}
where
\begin{equation*}
    A_j=\{i\in\{1,\dots,m\}: \|g_j-Tf_i\|_Y \leq (1-\varepsilon)^3r\}\quad\text{for $j\in\{1,\dots,k\}$.}
\end{equation*}
Indeed, if
\begin{equation*}
    \sum_{i\in A_j}\mu(E_i) < \frac{\tau}{k}\;\mu(S) \quad\text{for every $j\in\{1,\dots,k\}$,}
\end{equation*}
then, since each $i\in\{1,\dots,m\}$ belongs to some $A_j$, one would have
\begin{equation*}
    \sum_{i=1}^{m} \mu(E_i) \le \sum_{j=1}^{k}\sum_{i\in A_j} \mu(E_i) < \sum_{j=1}^{k}\frac{\tau}{k}\mu(S)=\tau\mu(S),
\end{equation*}
which is impossible owing to \ref{item:large_measure}. We fix such a $j$.

We consider the case~\ref{item:small_m} first. Let $i\in A_j$. Set $F_i=\{x\in E_i\colon |g_j(x)|\ge \varepsilon s_i\}$.
We have $|Tf_i - g_j|\ge (1-\varepsilon)s_i$ on $E_i\setminus F_i$. Consequently, using \eqref{E:f*-alt}, \eqref{E:quasiconcave_majorant_equivalent_norms_m}, the definition of $A_j$, and~\ref{item:large_marc_norm}, we estimate
\begin{equation*}
    (1-\varepsilon)s_i\varphi(\mu(E_i\setminus F_i))\leq \|Tf_i-g_j\|_{m_{\varphi}} \le (1-\varepsilon)^3r \leq (1-\varepsilon)^2 s_i\varphi(\mu(E_i)).
\end{equation*}
Hence
\begin{equation*}
    \varphi(\mu(E_i\setminus F_i))\leq (1-\varepsilon)\varphi(\mu(E_i)).
\end{equation*}
 Note that, since $t\mapsto \frac{\varphi(t)}{t}$ is nonincreasing on the interval $(0, t_0)$, it follows that
\begin{equation*}
    \mu(E_i\setminus F_i)\le (1-\varepsilon)\mu(E_i),
\end{equation*}
whence
\begin{equation}\label{E:general-lower:eq1}
    \mu(F_i)\ge \varepsilon\mu(E_i).
\end{equation}
Indeed,
\begin{equation*}
(1-\varepsilon) \geq \frac{\varphi(\mu(E_i\setminus F_i))}{\varphi(\mu(E_i))} = \frac{\varphi(\mu(E_i\setminus F_i))}{\mu(E_i\setminus F_i)}\frac{\mu(E_i)}{\varphi(\mu(E_i))} \frac{\mu(E_i\setminus F_i)}{\mu(E_i)}\geq \frac{\mu(E_i\setminus F_i)}{\mu(E_i)}.
\end{equation*}

Now, set $F = \bigcup_{i \in A_j} F_i$. Since the sets $F_i$, $i\in A_j$, are disjoint, we have
\begin{equation}\label{E:general-lower:eq2}
    \mu(F) \geq \varepsilon \sum_{i\in A_j} \mu(E_i)
\end{equation}
thanks to \eqref{E:general-lower:eq1}. Set $s = \min_{i \in A_j} s_i$, and note that $s\geq \sigma$ thanks to~\ref{item:uniform_bound_from_below}. Finally, using~\eqref{E:quasiconcave_majorant_equivalent_norms_m}, the fact that $g_j\geq s$ $\mu$-a.e.~in $F$, \eqref{E:general-lower:eq2}, \eqref{E:new-star}, and the fact that $\tff$ is nondecreasing, we obtain
\begin{align*}
   \frac1{C_\varphi} \|g_j\|_{m_{\varphi}} &\geq \|g_j\|_{m_{\tff}} \geq \|g_j\chi_F\|_{m_{\tff}} \geq \varepsilon s \tff(\mu(F)) \\
    &\ge \varepsilon s\tff\Big(\varepsilon\sum_{i\in A_j}\mu(E_i)\Big)\ge \varepsilon \sigma \tff\big(\varepsilon \tfrac{\tau}{k}\;\mu(S)\big).
\end{align*}
However, this combined with \eqref{E:P:general-lower:choice_sigma_small_m} contradicts \eqref{EE:out}. Therefore, $\alpha(T)\ge r$.

It remains to consider the case~\ref{item:capital_M}. Let $i\in A_j$. Using~\eqref{E:f**-alt}, \eqref{E:quasiconcave_majorant_equal_norms_M}, and the assumption \ref{item:large_marc_norm}, we have
\begin{align*}
    \frac{\tff(\mu(E_i))}{\mu(E_i)}\int_{E_i}|Tf_i-g_j| \, \d\mu &\le\|Tf_i-g_j\|_{M_{\tff}} = \|Tf_i-g_j\|_{M_{\varphi}} \\
    &\le (1-\varepsilon)^3r \leq (1-\varepsilon)^2s_i\varphi(\mu(E_i)) \\
    &\leq (1-\varepsilon)s_i\tff(\mu(E_i)).
\end{align*}
Hence
\begin{equation*}
    \frac1{\mu(E_i)}\int_{E_i}|Tf_i-g_j| \, \d\mu \le (1-\varepsilon)s_i.
\end{equation*}
Consequently,
\begin{align*}
    \frac1{\mu(E_i)}\int_{E_i}|g_j| \, \d\mu &\ge \frac1{\mu(E_i)}\int_{E_i}|Tf_i| \, \d\mu -\frac1{\mu(E_i)}\int_{E_i}|Tf_i-g_j| \, \d\mu \\
    &\ge s_i-(1-\varepsilon)s_i = \varepsilon s_i\ge \varepsilon \sigma.
\end{align*}
Now, set $E=\bigcup_{i \in A_j}E_i$, and note that
\begin{equation*}
    \int_{E}|g_j| \, \d\mu =\sum_{i \in A_j} \mu(E_i)\frac1{\mu(E_i)}\int_{E_i}|g_j| \, \d\mu \ge
    \varepsilon \sigma \sum_{i \in A_j} \mu(E_i) = \varepsilon \sigma \mu(E).
\end{equation*}
Hence
\begin{equation}\label{E:general-lower:eq3}
    g_j^{**}(\mu(E)) \geq \varepsilon \sigma
\end{equation}
thanks to \eqref{E:f**-alt}.
Finally, using~\eqref{E:quasiconcave_majorant_equal_norms_M}, \eqref{E:general-lower:eq3}, \eqref{E:new-star}, and the fact that $\tff$ is nondecreasing, we obtain
\begin{equation*}
    \|g_j\|_{M_{\varphi}} = \|g_j\|_{M_{\tff}} \ge \varepsilon \sigma \tff(\mu(E)) \ge \varepsilon \sigma \tff(\tfrac{\tau}k\mu(S)).
\end{equation*}
However, this combined with \eqref{E:P:general-lower:choice_sigma_capital_M} contradicts \eqref{EE:out}. Hence, once again, we obtain that $\alpha(T)\ge r$, as desired.
\end{proof}

In the following application of Theorem~\ref{T:general-lower}, $(\RRR,\mu)$ is an open subset of $\R^n$ endowed with the Lebesgue measure and $T=I$ is the embedding operator. First, we need a definition. We say that a quasinormed space $X$ of measurable functions on an open set $\Omega\subseteq\rn$ is \emph{translation invariant} if, for every $f\in X$ and $y\in\rn$ such that $y + \spt f \subseteq \Omega$, the function $g(x) = f(x - y)\chi_{y + \spt f}(x)$, $x\in\Omega$, is also in $X$ and $\|g\|_X = \|f\|_X$. We note that many customary spaces, including Sobolev ones, possess this property. 

\begin{theorem} \label{T:max-noncompact}
Let $n\in\N$ and $\Omega\subseteq\rn$ be an open set. Assume that $\varphi\colon(0, |\Omega|) \to (0, \infty)$ and $Y$ are as in \ref{item:small_m} or \ref{item:capital_M} in Theorem~\ref{T:general-lower}. Furthermore, assume that $\varphi$ satisfies
\begin{equation}\label{E:varphi-vanishing-at-0}
    \lim_{t\to0_+}\varphi(t)=0.
\end{equation}
Let $X$ be a quasinormed space of measurable functions defined on $\Omega$ that is translation invariant and such that $X\subseteq Y$. Let $\alpha_0\in(0, \|I\|)$, where $I\colon X \to Y$ is the embedding operator. Assume that there are a point $x_0\in\Omega$ and a sequence of functions $\{f_j\}_{j=1}^\infty \subseteq X$ such that, for every $j\in\N$, $\|f_j\|_X = 1$, the support of $f_j$ is inside a ball $B_j\subseteq\Omega$ centered at $x_0$ with radius $r_j\searrow 0$, $f_j$ is radially nonincreasing with respect to $x_0$, and  $\| f_j \|_Y \geq \alpha_0$. Then $\alpha(I) \geq \alpha_0$.

In particular, if such a sequence exists for every $\alpha_0\in(0, \|I\|)$, then the embedding $I\colon X\to Y$ is maximally noncompact.
\end{theorem}

\begin{proof}
Since $\Omega$ is open, there is a (closed) cube $Q\subseteq\Omega$ centered at $x_0$ whose edges are parallel to the coordinate axes. Furthermore, we may assume that $Q$ is so small that $x - x_0 + Q \subseteq \Omega$ for every $x\in Q$. Let $B_0$ be the ball inscribed in $Q$. Set
\begin{equation}\label{E:T:max-noncompact:tau}
\tau = \frac{|B_0|}{2^n|Q|}.
\end{equation}
Let $\mathfrak Q_k$, $k\in\N_0$, be the $k$th subdivision of $Q$ consisting of the $2^{kn}$ nonoverlaping subcubes of $Q$ whose measure is $\frac{|Q|}{2^{kn}}$.
Let $\alpha_0\in(0, \|I\|)$. Fix arbitrary $\varepsilon\in(0, 1)$ and $\sigma > 0$. Thanks to \eqref{E:varphi-vanishing-at-0}, there is $t_0\in(0, |\Omega|)$ such that
\begin{equation}\label{E:T:max-noncompact:t0}
\frac{(1 - \varepsilon)\alpha_0}{\varphi(t)} > \sigma \quad \text{for every $t\in(0, t_0)$},
\end{equation}
and which in the case \ref{item:small_m} is so small that the function $\frac{\varphi(t)}{t}$ is nonincreasing on the interval $(0,t_0)$. Furthermore, we can find $j\in\N$ such that $f = f_j\in B_X$ is supported inside a ball $B = B_j \subseteq B_0$ centered at $x_0$ with $|B| < t_0$, $f$ is radially nonincreasing with respect to $x_0$, and $\| f \|_Y \geq \alpha_0$. It follows from $\|f\|_{Y} \geq \alpha_0 > (1 - \varepsilon)\alpha_0$ combined with \eqref{E:f*-alt} or \eqref{E:f**-alt} that there is $t_1\in(0, |B|)$ such that
\begin{equation*}
\varphi(t_1) \sup_{|E| = t_1}s_{E, f} > (1 - \varepsilon)\alpha_0,
\end{equation*}
where
\begin{equation*}
s_{E, f} = \begin{cases}
\essinf_{x\in E} |f(x)| \quad &\text{if $Y = m_\varphi(\Omega)$}, \\
\frac1{|E|} \int_E |f(x)| \, \d x \quad &\text{if $Y = M_\varphi(\Omega)$}.
\end{cases}
\end{equation*}
In both cases, there is a set $E\subseteq B\subseteq B_0$ with $|E| = t_1 < t_0$ such that
\begin{align}
\varphi(t_1) s_{E,f} > (1 - \varepsilon)\alpha_0 \label{E:T:max-noncompact:(iv)satisfied},
    \\
\intertext{that is}
s_{E,f}  > \frac{(1 - \varepsilon)\alpha_0}{\varphi(t_1)} > \sigma, \label{E:T:max-noncompact:(iii)satisfied}
\end{align}
where the last inequality follows from \eqref{E:T:max-noncompact:t0}. Furthermore, since $f$ is radially nonincreasing with respect to $x_0$, the set $E$ may be chosen as a ball centered at $x_0$. Let $Q_E$ be the cube in which $E$ is inscribed. Note that
\begin{equation}\label{E:volume_ratio_inscribed_ball}
    \frac{|Q_E|}{|E|} = \frac{|Q|}{|B_0|} = \frac{2^n}{\omega_n},
\end{equation}
where $\omega_n$ is the volume of the unit ball in $\rn$.

Next, let $k\in\N_0$ be the unique integer such that
\begin{equation}\label{E:T:max-noncompact:level_of_subdivision}
\frac{|B_0|}{2^{n(k+1)}} \leq t_1 < \frac{|B_0|}{2^{nk}}.
\end{equation}
Let $\mathfrak Q_k = \{Q_1, \dots, Q_m\}$, where $m = 2^{nk}$. For each $j\in\{1,\dots,m\}$, denote by $x_{Q_j}$ the center of the cube $Q_j$. We define the functions $g_j$ and the sets $E_j$ as
\begin{align*}
g_j(x) &= f(x + x_0 - x_{Q_j})\chi_{x_{Q_j} - x_0 + \spt f}(x),\ x\in\Omega,\\
\intertext{and}
E_j &= x_{Q_j} - x_0 + E
\end{align*}
for $j=1, \dots, m$.
Now, it is easy to see that $s_{E_j, g_j} = s_{E,f}$ for every $j\in\{1, \dots, m\}$. Therefore, both \eqref{E:T:max-noncompact:(iv)satisfied} and \eqref{E:T:max-noncompact:(iii)satisfied} are satisfied with $s_{E,f}$ replaced by $s_{E_j, g_j}$. Furthermore, $\|g_j\|_X = \|f\|_X = 1$ for every $j\in\{1, \dots, m\}$ thanks to the translation invariance of $X$.
Finally, we claim that the sets $E_j$, $j=1, \dots, m$, are pairwise disjoint. To this end, note that, for every $j=1, \dots, m$, the ball $E_j$ is inscribed in the cube $x_{Q_j} - x_0 + Q_E$ and we have
\begin{equation*}
    |Q_E| = \frac{2^n}{\omega_n} t_1 < \frac{2^n}{\omega_n} \frac{|B_0|}{2^{nk}} = \frac{|Q|}{2^{nk}} = |Q_j|
\end{equation*}
thanks to \eqref{E:volume_ratio_inscribed_ball} and \eqref{E:T:max-noncompact:level_of_subdivision}. Since the cubes $x_{Q_j} - x_0 + Q_E$ and $Q_j$ are concentric, it follows that $E_j\subseteq x_{Q_j} - x_0 + Q_E \subseteq Q_j$. Consequently, since the cubes in $\mathfrak Q_k$ are nonoverlapping, the sets $E_j$, $j=1, \dots, m$, are pairwise disjoint. Moreover, thanks to \eqref{E:T:max-noncompact:tau} and \eqref{E:T:max-noncompact:level_of_subdivision}, we have
\begin{equation*}
\sum_{j = 1}^m |E_j| = m |E| = 2^{nk} t_1 \geq \frac{2^{nk}|B_0|}{2^{n(k+1)}} = \frac{|B_0|}{2^n} = \tau |Q|.
\end{equation*}
Therefore, the assumptions of Theorem~\ref{T:general-lower} are satisfied (with $T = I$, $S = Q$, and with $r$ and $\{f_i\}_{i = 1}^m$ there equal to $\alpha_0$ and $\{g_i\}_{i = 1}^m$, respectively), whence it follows that $\alpha(I) \geq \alpha_0$.
\end{proof}

As a corollary of the preceding theorem, we obtain that embeddings with a certain shrinking property into Marcinkiewicz spaces are maximally noncompact.
\begin{definition}
Let $n\in\N$ and let $X$ and $Y$ be quasinormed spaces of measurable functions defined on an open set $\Omega\subseteq\rn$.
We say that the embedding $I\colon X\to Y$ has the \textit{shrinking property} if there is a point $x_0\in\Omega$ such that $\nrm{I_{B}} = \nrm{I}$ for every open ball $B\subseteq\Omega$ centered at $x_0$, where
\begin{equation*}
	\nrm{I_{B}} = \sup\left\{\frac{\|f\|_{Y}}{\|f\|_{X}}: \ \spt u \subseteq B\ \text{and $f$ is radially nonincreasing with respect to $x_0$}\right\}.
\end{equation*}
\end{definition}

\begin{corollary}\label{C:max-noncompact_shrinking}
Let $\Omega$, $\varphi$, $X$, and $Y$ be as in Theorem~\ref{T:max-noncompact}. If the embedding $I\colon X \to Y$ has the shrinking property, then it is maximally noncompact.
\end{corollary}

We conclude this section by providing a certain alternative approach to the maximal noncompactness of embeddings into Marcinkiewicz spaces.  The following theorem generalizes~\cite[Theorem~4.1]{Lan:20}, which is limited to power functions. Compared to Theorem~\ref{T:max-noncompact}, it does not require $X$ to be translation invariant, nor do the extremals need to be radially nonincreasing. Furthermore, unlike in the rest of this section, the function $\varphi$ need not be even admissible (still less almost quasiconcave) when $Y = m_\varphi$. For example, the function $\varphi(t) = t^\frac1{p}$ for $p \in (0,1)$ is not almost quasiconcave (and not admissible either when $\mu(\RRR) = \infty$), but it satisfies the assumptions of the following theorem with $Y = m_\varphi$ nonetheless. Recall that $m_\varphi = L^{p,\infty}$ is the weak Lebesgue space for this choice of $\varphi$. On the other hand, in the following theorem, the extremals need to be equimeasurable and the assumption~\eqref{T:maximal_noncomp_alt_m_phi:not_too_constant} below requires $\varphi$ to grow sufficiently fast (or rather, $\varphi$ cannot be too slowly varying). A typical example of an almost-quasiconcave function not satisfying the assumption~\eqref{T:maximal_noncomp_alt_m_phi:not_too_constant} is the function $\varphi(t) = \log(2\mu(\RRR)/t)^\alpha$ for $\alpha\leq0$, assuming $\mu(\RRR) < \infty$. Therefore, while the following theorem and Theorem~\ref{T:max-noncompact} overlap to some extent, they are noncomparable in general and complement each other. Moreover, even where they overlap, it may sometimes be easier to verify the assumptions of one than those of the other.

We say that two measurable functions are \emph{equimeasurable} if their nonincreasing rearrangements coincide. The following simple observation, which follows from the definition of the nonincreasing rearrangement, will be useful in the proof below. For every $N\in\N$ and every collection of measurable functions $\{f_j\}_{j = 1}^N$ with mutually disjoint supports, we have
\begin{equation}\label{E:rearrangement_sum_disjoint_supports_lower_bound}
    \Big( \sum_{j = 1}^N f_j\Big)^*(N t) \geq \min_{j = 1, \dots, N} f_j^*(t) \quad \text{for every $t\in(0, \infty)$}.
\end{equation}
\begin{theorem}\label{T:maximal_noncomp_alt_m_phi}
Let $X$ be a quasinormed space of measurable functions on $(\RRR, \mu)$. Let $\varphi\colon (0, \mu(R)) \to (0, \infty)$ be a function satisfying
\begin{equation}\label{T:maximal_noncomp_alt_m_phi:not_too_constant}    
\lim_{a \to \infty} \inf_{t\in(0, \mu(R)/a)} \frac{\varphi(a t)}{\varphi(t)} = \infty.
\end{equation}
Assume one of the following:
\begin{enumerate}[label={\rm(\alph*)},ref={\rm(\alph*)},leftmargin=6ex]
    \item\label{item:small_m_alt} $Y = m_\varphi(R, \mu)$, $\varphi\in\Delta_2$, and
    \begin{equation}\label{T:maximal_noncomp_alt_m_phi:dilation}
        \inf_{\theta\in(0,1)} \sup_{t\in(0, \mu(R))} \frac{\varphi(t)}{\varphi(\theta t)} \leq 1;
    \end{equation}
    \item\label{item:capital_M_alt} $Y = M_\varphi(R, \mu)$ and $\varphi$ is admissible.
\end{enumerate}
Let $T\colon X \to Y$ be a bounded positively homogeneous operator. Let $\lambda\in(0, \|T\|)$.

Assume that, for every $a\in(0, \mu(R))$ and $M\in\N$, there is a collection of functions $\{f_j\}_{j = 1}^M \subseteq B_X$ such that
\begin{align}
&\text{the supports of $Tf_j$, $j=1, \dots, M$, are mutually disjoint}, \label{T:maximal_noncomp_alt_m_phi:extremals_disj_spt}\\
&\mu(\spt Tf_j) \leq a \quad \text{for every $j=1, \dots, M$}, \label{T:maximal_noncomp_alt_m_phi:extremals_small_spt}\\
&\text{the functions $Tf_j$, $j=1, \dots, M$, are equimeasurable}, \label{T:maximal_noncomp_alt_m_phi:extremals_equimea}\\
\intertext{and that}
&\| Tf_j \|_{Y} \geq \lambda. \label{T:maximal_noncomp_alt_m_phi:extremals_extremality}
\end{align}
Then $\bmc(T)\geq \lambda$. In particular, if such a collection of functions $\{f_j\}_{j = 1}^M$ exists for every $\lambda\in(0, \|T\|)$, then the operator $T$ is maximally noncompact.
\end{theorem}

\begin{proof}
Set $\alpha = \bmc(T)$. Suppose that $\alpha < \lambda$, and fix any $r,\varepsilon>0$ such that
\begin{equation*}
\alpha < r < r + 2\varepsilon < \lambda.
\end{equation*}
Since $\alpha < r$, there are $m\in\N$ and functions $\{g_k\}_{k = 1}^m \subseteq Y$ such that
\begin{equation}\label{T:maximal_noncomp_alt_m_phi:union}
T(B_X) \subseteq \bigcup_{k = 1}^m \big( g_k + r B_{Y} \big).
\end{equation}
Note that we may assume that
\begin{equation}  
\label{T:maximal_noncomp_alt_m_phi:centers_not_too_big_qnorm}
    \|g_k\|_{Y} \leq 2 C \|T\| \quad \text{for every $k\in \{1, \dots, m\}$},
\end{equation}
where either $C = 1$ when $Y = M_\varphi$ or $C$ is equal to the $\Delta_2$ constant of $\varphi$, i.e.,
\begin{equation*}
  C = \sup_{t\in(0, \mu(R)/2)} \frac{\varphi(2t)}{\varphi(t)} < \infty,
\end{equation*}
when $Y = m_\varphi$. 
Indeed, if there is $k\in \{1, \dots, m\}$ such that $\|g_k\|_Y > 2 C \|T\|$, then, by Lemma~\ref{L:distance}, 
\begin{equation*}
    T(B_X) \cap \big( g_k + r B_{Y} \big)
    =\emptyset,
\end{equation*}
and we can exclude such a $g_k$ from the union in \eqref{T:maximal_noncomp_alt_m_phi:union}. Here we used the fact that $\|\cdot\|_{M_\varphi}$ is a norm, whereas $\|\cdot\|_{m_\varphi}$ is a quasinorm with $C_{m_\varphi}$ equal to $C$.

We will sometimes need to distinguish among three possibilities. Note that one of the following (mutually exclusive) three possibilities is true:
\begin{enumerate}[label=(C\arabic*), ref=(C\arabic*)]
\item\label{item:possibility_small_m_or_infinite_measure} either $Y = M_\varphi$ and $\mu(\RRR) = \infty$ or $Y = m_\varphi$;
\item\label{item:possibility_capital_M_not_identity} $Y = M_\varphi$, $\mu(\RRR) < \infty$, and for every $\alpha\in(0, \mu(\RRR))$ there is $\beta_0\in(0, \alpha)$ such that
\begin{equation*}
    \frac{\tff(\beta_0)}{\beta_0} > \frac{\tff(\alpha)}{\alpha};
\end{equation*}
\item\label{item:possibility_capital_M_identity} $Y = M_\varphi$, $\mu(\RRR) < \infty$, and there is $\alpha_0\in(0, \mu(\RRR))$ such that for every $\beta\in(0, \alpha_0)$
\begin{equation}\label{T:maximal_noncomp_alt_m_phi:identity_case_constantness}
    \frac{\tff(\beta)}{\beta} = \frac{\tff(\alpha_0)}{\alpha_0}.
\end{equation}
\end{enumerate}
Here (and below), $\tff$ is the least quasiconcave majorant of $\varphi$. If \ref{item:possibility_capital_M_identity} is the case, we fix $\alpha_0$ at this point. As an aside, we remark that the need to distinguish among \ref{item:possibility_small_m_or_infinite_measure}, \ref{item:possibility_capital_M_not_identity}, and \ref{item:possibility_capital_M_identity} is due to the fact that $\varphi$ is merely admissible in the case~\ref{item:capital_M_alt}, not necessarily almost quasiconcave (see~Remark~\ref{rem:maximal_noncomp_alt_m_phi:phi_quasiconcave} below).

In the case~\ref{item:small_m_alt},  there is $\theta\in(0,1)$ such that
\begin{equation}\label{T:maximal_noncomp_alt_m_phi:dilation_chosen_theta}
\sup_{t\in(0, \mu(R))} \frac{\varphi(t)}{\varphi(\theta t)} \leq 1 + \frac{\varepsilon}{r}
\end{equation}
thanks to \eqref{T:maximal_noncomp_alt_m_phi:dilation}. In the case~\ref{item:capital_M_alt}, we simply set $\theta = 0$. Either way, if \ref{item:possibility_small_m_or_infinite_measure} or \ref{item:possibility_capital_M_not_identity} is the case, we fix $N\in\N$ such that
\begin{equation}\label{T:maximal_noncomp_alt_m_phi:not_too_constant_chosen_theta}
\inf_{t\in(0, \mu(R)/(N(1-\theta)))} \frac{\varphi(N(1-\theta) t)}{\varphi(t)} > \frac{2 C}{\varepsilon}\|T\|,
\end{equation}
which is possible owing to \eqref{T:maximal_noncomp_alt_m_phi:not_too_constant}. If \ref{item:possibility_capital_M_identity} is the case, we fix $N\in\N$, $N\geq2$, so large that
\begin{equation}\label{T:maximal_noncomp_alt_m_phi:identity_case_choice_of_N}
    \frac{\mu(\RRR)}{N} < \alpha_0 \quad \text{and} \quad N> \frac{2\|T\|}{\varepsilon}.
\end{equation}
Next, if \ref{item:possibility_capital_M_not_identity} is the case (in particular, $\mu(\RRR) < \infty$), we fix $\beta_0\in(0, \frac{\mu(\RRR)}{N})$, for $\alpha = \frac{\mu(\RRR)}{N}$,  such that
\begin{equation}\label{T:maximal_noncomp_alt_m_phi:non_identity_jump}
        \frac{\tff(\beta_0)}{\beta_0} > \frac{\tff(\frac{\mu(\RRR)}{N})}{\frac{\mu(\RRR)}{N}}.
\end{equation}
Furthermore, let $\gamma > 1$ be such that
\begin{equation}\label{T:maximal_noncomp_alt_m_phi:non_identity_beta0}
    \beta_0 = \frac{\mu(\RRR)}{\gamma N}.
\end{equation}
If \ref{item:possibility_small_m_or_infinite_measure} or \ref{item:possibility_capital_M_identity} is the case, we set $\gamma = 1$. Now, set
\begin{align}
a &= \frac{\min\{\mu(R), 1\}}{\gamma N^2(1-\theta)} \leq \frac{\min\{\mu(R), 1\}}{N(1-\theta)} \label{T:maximal_noncomp_alt_m_phi:a_small_enough}\\
\intertext{and}
M &= mN. \nonumber
\end{align}
Let $\{f_j\}_{j = 1}^M\subseteq B_X$ be a collection of functions satisfying \eqref{T:maximal_noncomp_alt_m_phi:extremals_disj_spt}--\eqref{T:maximal_noncomp_alt_m_phi:extremals_extremality}. Note that
\begin{equation*}
\{Tf_j\}_{j = 1}^{M} \subseteq \bigcup_{k = 1}^m \big( g_k + r B_Y \big)
\end{equation*}
thanks to \eqref{T:maximal_noncomp_alt_m_phi:union} combined with the fact that $\{f_j\}_{j = 1}^M\subseteq B_X$. Therefore, by the pigeonhole principle, there has to be at least one $k_0 \in \{1, \dots, m\}$ such that
at least $M/m = N$ functions from $\{Tf_j\}_{j = 1}^{M}$ are contained in $g_{k_0} + r B_Y$. Clearly, we may assume that $\{Tf_1, \dots, Tf_N\}$ are such functions\textemdash otherwise, we would re-index the collection. Hence
\begin{equation}\label{T:maximal_noncomp_alt_m_phi:union_pigeonhole}
\|Tf_j - g_{k_0}\|_Y \leq r \quad \text{for every $j\in\{1, \dots, N\}$}.
\end{equation}
Next, for every $j\in\{1, \dots, N\}$, we define the functions $F_j$ and $h_j$ as
\begin{align*}
h_j(x) &= g_{k_0}(x) \chi_{\spt Tf_j}(x) \\
\intertext{and}
F_j(x) &= \begin{cases}
	h_j(x) \quad &\text{if $|h_j(x)| \leq |Tf_j(x)|$},\\
	Tf_j(x) \quad &\text{if $|Tf_j(x)| \leq |h_j(x)|$},
\end{cases}
\end{align*}
for every $x\in\RRR$.
Note that $\spt h_j \cup \spt F_j \subseteq \spt Tf_j$ for every $j\in\{1, \dots, N\}$. Furthermore, using \eqref{T:maximal_noncomp_alt_m_phi:extremals_disj_spt}, we observe that
\begin{equation*}
\Big| \sum_{j = 1}^N F_j \Big| \leq \Big| \sum_{j = 1}^N h_j \Big| \leq |g_{k_0}| \quad \text{$\mu$-a.e.~in $\RRR$}.
\end{equation*}
Hence
\begin{equation}\label{T:maximal_noncomp_alt_m_phi:qnorm_sum_fj_upper}
\Big\| \sum_{j = 1}^N F_j \Big\|_{Y} \leq \|g_{k_0}\|_{Y} \leq 2C\|T\|
\end{equation}
thanks to \eqref{T:maximal_noncomp_alt_m_phi:centers_not_too_big_qnorm}. Since $|F_j - Tf_j| \leq |h_j - Tf_j| \leq |g_{k_0} - Tf_j|$ $\mu$-a.e.~in $\RRR$ for every $j\in\{1, \dots, N\}$, we also have
\begin{equation}\label{T:maximal_noncomp_alt_m_phi:fj_and_Txj_close}
\|F_j - Tf_j\|_{Y} \leq \|g_{k_0} - Tf_j\|_{Y} \leq r \quad \text{for every $j\in\{1, \dots, N\}$}
\end{equation}
owing to \eqref{T:maximal_noncomp_alt_m_phi:union_pigeonhole}. Aiming to reach a contradiction with \eqref{T:maximal_noncomp_alt_m_phi:qnorm_sum_fj_upper}, we now need to distinguish between the cases~\ref{item:small_m_alt} and~\ref{item:capital_M_alt}.

First, we consider the case~\ref{item:small_m_alt} (and so \ref{item:possibility_small_m_or_infinite_measure} is the case). Note that
\begin{equation*}
\|Tf_j\|_{Y} = \|Tf_j\|_{m_\varphi} = \sup_{t\in(0, a)} (Tf_j)^*(t) \varphi(t)
\end{equation*}
thanks to \eqref{T:maximal_noncomp_alt_m_phi:extremals_small_spt}, and that
\begin{equation*}
(Tf_j)^* = (Tf_i)^* \quad \text{for every $i,j\in\{1, \dots, N\}$}
\end{equation*}
thanks to \eqref{T:maximal_noncomp_alt_m_phi:extremals_equimea}. Hence, combining these two identities with \eqref{T:maximal_noncomp_alt_m_phi:extremals_extremality} and with the fact that $r + 2\varepsilon < \lambda$, we obtain the existence of $t_0\in(0,a)$ such that
\begin{equation}\label{T:maximal_noncomp_alt_m_phi:point_of_extremality_small_m}
(Tf_j)^*(t_0) \geq \frac{r + 2\varepsilon}{\varphi(t_0)} \quad \text{for every $j\in\{1, \dots, N\}$}.
\end{equation}
Next, we claim that
\begin{equation*}
(F_j)^*((1-\theta) t_0) \geq \frac{\varepsilon}{\varphi(t_0)} \quad \text{for every $j\in\{1, \dots, N\}$}.
\end{equation*}
To this end, we need to observe two things. First, note that it follows from \eqref{T:maximal_noncomp_alt_m_phi:dilation_chosen_theta} that
\begin{equation*}
r\Big(\frac1{\varphi(t_0)} - \frac1{\varphi(\theta t_0)} \Big) \geq -\frac{\varepsilon}{\varphi(t_0)}.
\end{equation*}
Second, we have
\begin{equation*}
(Tf_j - F_j)^*(\theta t_0) \leq \frac{\|Tf_j - F_j\|_{m_\varphi}}{\varphi(\theta t_0)} \leq \frac{r}{\varphi(\theta t_0)} \quad \text{for every $j\in\{1, \dots, N\}$}
\end{equation*}
thanks to \eqref{T:maximal_noncomp_alt_m_phi:fj_and_Txj_close}. Hence, combining these two observations with \eqref{T:maximal_noncomp_alt_m_phi:point_of_extremality_small_m} and using \eqref{E:f*-subadditivity}, we obtain
\begin{align*}
(F_j)^*((1 - \theta) t_0) &\geq (Tf_j)^*(t_0) - (Tf_j - F_j)^*(\theta t_0)\\
&\geq \frac{r + 2\varepsilon}{\varphi(t_0)} - \frac{r}{\varphi(\theta t_0)}\\
&= \frac{2\varepsilon}{\varphi(t_0)} + r\Big(\frac1{\varphi(t_0)} - \frac1{\varphi(\theta t_0)} \Big) \nonumber\\
&\geq \frac{\varepsilon}{\varphi(t_0)}
\end{align*}
for every $j\in\{1, \dots, N\}$. Finally, since the functions $\{F_j\}_{j = 1}^N$ have mutually disjoint supports, the last estimate combined with \eqref{E:rearrangement_sum_disjoint_supports_lower_bound} implies that
\begin{equation*}
\Big( \sum_{j = 1}^N F_j \Big)^*(N(1 - \theta) t_0) \geq \frac{\varepsilon}{\varphi(t_0)}.
\end{equation*}
Combining this with \eqref{T:maximal_noncomp_alt_m_phi:not_too_constant_chosen_theta}, we arrive at
\begin{align*}
\Big\| \sum_{j = 1}^N F_j \Big\|_{m_\varphi} &\geq \Big( \sum_{j = 1}^N F_j \Big)^*(N(1 - \theta) t_0) \varphi(N(1 - \theta) t_0) \\
&\geq \varepsilon\frac{\varphi(N(1 - \theta) t_0)}{\varphi(t_0)} >  2C\|T\|.
\end{align*}
However, this contradicts \eqref{T:maximal_noncomp_alt_m_phi:qnorm_sum_fj_upper}. Therefore, $\bmc(T) \geq \lambda$.

It remains to consider the case~\ref{item:capital_M_alt}. Similarly to the case~\ref{item:small_m_alt}, note that
\begin{equation}\label{T:maximal_noncomp_alt_m_phi:capital_M_eq1}
(Tf_j)^{**} = (Tf_i)^{**} \quad \text{for every $i,j\in\{1, \dots, N\}$}
\end{equation}
thanks to \eqref{T:maximal_noncomp_alt_m_phi:extremals_equimea}. Furthermore, we claim that
\begin{equation}\label{T:maximal_noncomp_alt_m_phi:capital_M_eq2}
\|Tf_j\|_{M_\varphi} = \sup_{t\in(0, a)} \tff(t) (Tf_j)^{**}(t)  \quad \text{for every $j\in\{1, \dots, N\}$}.
\end{equation}
To this end, thanks to \eqref{T:maximal_noncomp_alt_m_phi:extremals_small_spt} and the monotonicity of the function $t\mapsto \tff(t)/t$, we have
\begin{equation*}
    \sup_{t\in[a, \mu(\RRR))} \tff(t) (Tf_j)^{**}(t) = \Big(\int_0^a (Tf_j)^* \Big) \sup_{t\in[a, \mu(\RRR))} \frac{\tff(t)}{t} = \tff(a) (Tf_j)^{**}(a) \leq \sup_{t\in(0, a)} \tff(t) (Tf_j)^{**}(t),
\end{equation*}
where we also used the continuity of the functions $\tff$ and $(Tf_j)^{**}$ in the inequality. Combining this with \eqref{E:quasiconcave_majorant_equal_norms_M}, we obtain
\begin{align*}
    \|Tf_j\|_{M_\varphi} &= \|Tf_j\|_{M_{\tff}} = \max\Big\{\sup_{t\in(0, a)} \tff(t) (Tf_j)^{**}(t), \sup_{t\in[a, \mu(\RRR))} \tff(t) (Tf_j)^{**}(t) \Big\} \\
    &= \sup_{t\in(0, a)} \tff(t) (Tf_j)^{**}(t)
\end{align*}
for every $j\in\{1, \dots, N\}$. Hence \eqref{T:maximal_noncomp_alt_m_phi:capital_M_eq2} is true. Therefore, it follows from \eqref{T:maximal_noncomp_alt_m_phi:extremals_extremality} combined with \eqref{T:maximal_noncomp_alt_m_phi:capital_M_eq2} and \eqref{T:maximal_noncomp_alt_m_phi:capital_M_eq1} that there is
$t_0\in(0,a)$ such that
\begin{equation*}
(Tf_j)^{**}(t_0) > \frac{r + \varepsilon}{\tff(t_0)} \quad \text{for every $j\in\{1, \dots, N\}$},
\end{equation*}
where we also used the fact that $r + \varepsilon < \lambda$. Moreover, thanks to this and \eqref{E:f**-alt}, there are sets $E_j\subseteq \spt Tf_j$, $j = 1, \dots, M$, such that $\mu(E_j) = t_0$ and
\begin{equation}\label{T:maximal_noncomp_alt_m_phi:point_of_extremality_capital_M}
    \frac1{t_0} \int_{E_j} |Tf_j| \, \d\mu > \frac{r + \varepsilon}{\tff(t_0)} \quad \text{for every $j\in\{1, \dots, N\}$}.
\end{equation}
Next, for future reference, note that the sets $E_j$, $j = 1, \dots, N$, are disjoint thanks to \eqref{T:maximal_noncomp_alt_m_phi:extremals_disj_spt}, and so we have
\begin{equation}\label{T:maximal_noncomp_alt_m_phi:capital_M_eq3}
    \mu\Big( \bigcup_{j = 1}^N E_j \Big) = \sum_{j = 1}^N \mu(E_j) = Nt_0.
\end{equation}
Now, since
\begin{equation*}
    \frac1{t_0} \int_{E_j} |Tf_j - F_j| \, \d\mu \leq (Tf_j - F_j)^{**}(t_0) \quad \text{for every $j\in\{1, \dots, N\}$}
\end{equation*}
by \eqref{E:f**-alt}, we obtain
\begin{equation*}
    \frac{\tff(t_0)}{t_0} \int_{E_j} |Tf_j - F_j| \, \d\mu \leq \|Tf_j - F_j\|_{M_{\tff}} \leq r \quad \text{for every $j\in\{1, \dots, N\}$}
\end{equation*}
thanks to \eqref{T:maximal_noncomp_alt_m_phi:fj_and_Txj_close} and \eqref{E:quasiconcave_majorant_equal_norms_M}. Therefore, combining this with \eqref{T:maximal_noncomp_alt_m_phi:point_of_extremality_capital_M}, we arrive at
\begin{align*}
    \frac1{t_0} \int_{E_j} |F_j| \, \d\mu &\geq \frac1{t_0} \int_{E_j} |Tf_j| \, \d\mu - \frac1{t_0} \int_{E_j} |Tf_j - F_j| \, \d\mu \\
    &> \frac{r + \varepsilon}{\tff(t_0)}  - \frac{r}{\tff(t_0)} = \frac{\varepsilon}{\tff(t_0)}
\end{align*}
for every $j\in\{1, \dots, N\}$. Combining this with \eqref{E:f**-alt} and \eqref{T:maximal_noncomp_alt_m_phi:capital_M_eq3}, we obtain
\begin{equation*}
    \Big( \sum_{j = 1}^N F_j \Big)^{**}(Nt_0) \geq \frac1{Nt_0} \sum_{j = 1}^N \int_{E_j} |F_j| \,\d\mu \geq \frac1{N}\frac{\varepsilon N}{\tff(t_0)} = \frac{\varepsilon}{\tff(t_0)}.
\end{equation*}
This together with \eqref{E:quasiconcave_majorant_equal_norms_M} implies that
\begin{equation}\label{T:maximal_noncomp_alt_m_phi:capital_M_eq4}
    \Big\| \sum_{j = 1}^N F_j \Big\|_{M_\varphi} \geq \varepsilon\frac{\tff(Nt_0)}{\tff(t_0)}.
\end{equation}
Assume for the moment that
 \begin{equation}\label{T:maximal_noncomp_alt_m_phi:capital_M_eq5}
     \frac{\tff(Nt_0)}{\tff(t_0)} > \frac{2}{\varepsilon}\|T\|.
 \end{equation}
 Then, thanks to this and \eqref{T:maximal_noncomp_alt_m_phi:capital_M_eq4}, we obtain
 \begin{equation*}
     \Big\| \sum_{j = 1}^N F_j \Big\|_{M_\varphi} > 2\|T\|
 \end{equation*}
 again, which contradicts \eqref{T:maximal_noncomp_alt_m_phi:qnorm_sum_fj_upper} (recall that $C = 1$ in the case~\ref{item:capital_M_alt}). Consequently, $\bmc(T) \geq \lambda$. Therefore, the entire proof will be done once we establish \eqref{T:maximal_noncomp_alt_m_phi:capital_M_eq5}. To this end, we will need to distinguish among \ref{item:possibility_small_m_or_infinite_measure}, \ref{item:possibility_capital_M_not_identity}, and \ref{item:possibility_capital_M_identity}.
 
 First, assume that \ref{item:possibility_capital_M_identity} is the case. Note that
 \begin{equation*}
     t_0 \leq Nt_0 < N a \leq \frac{\mu(\RRR)}{N} < \alpha_0
 \end{equation*}
 thanks to \eqref{T:maximal_noncomp_alt_m_phi:a_small_enough} and \eqref{T:maximal_noncomp_alt_m_phi:identity_case_choice_of_N}. Hence
 \begin{equation*}
     \frac{\tff(Nt_0)}{\tff(t_0)} = \frac{\tff(Nt_0)}{Nt_0} \frac{t_0}{\tff(t_0)} N = N > \frac{2\|T\|}{\varepsilon}
 \end{equation*}
 owing to \eqref{T:maximal_noncomp_alt_m_phi:identity_case_constantness} and \eqref{T:maximal_noncomp_alt_m_phi:identity_case_choice_of_N} again. Hence \eqref{T:maximal_noncomp_alt_m_phi:capital_M_eq5} is true when \ref{item:possibility_capital_M_identity} is the case. Next, let us turn our attention to the remaining two possibilities, i.e., \ref{item:possibility_small_m_or_infinite_measure} or \ref{item:possibility_capital_M_not_identity}. Note that
 \begin{align}
     \tff(Nt_0) &=  Nt_0 \sup_{s\in[Nt_0, \mu(\RRR))} \frac{\sup_{\tau\in(0, s]} \varphi(\tau)}{s} = t_0 \sup_{s\in[Nt_0, \mu(\RRR))} \frac{\sup_{\tau\in(0, \frac{s}{N}]} \varphi(N\tau)}{\frac{s}{N}} \nonumber\\
     &= t_0 \sup_{s\in[t_0, \frac{\mu(\RRR)}{N})} \frac{\sup_{\tau\in(0, s]} \varphi(N\tau)}{s} > \frac{2\|T\|}{\varepsilon} t_0 \sup_{s\in[t_0, \frac{\mu(\RRR)}{N})} \frac{\sup_{\tau\in(0, s]} \varphi(\tau)}{s} \label{T:maximal_noncomp_alt_m_phi:capital_M_eq6}
 \end{align}
 where in the inequality we used \eqref{T:maximal_noncomp_alt_m_phi:not_too_constant_chosen_theta} (recall that $\theta = 0$ when \ref{item:possibility_capital_M_identity} is the case) combined with the fact that $\tau \leq s < \frac{\mu(\RRR)}{N}$. Now, if $\mu(\RRR) = \infty$ (i.e., \ref{item:possibility_small_m_or_infinite_measure} is the case), then $\mu(\RRR) / N = \mu(\RRR)$, and so
 \begin{equation*}
     \tff(Nt_0) > \frac{2\|T\|}{\varepsilon} t_0 \sup_{s\in[t_0, \mu(\RRR))} \frac{\sup_{\tau\in(0, s]} \varphi(\tau)}{s} = \frac{2\|T\|}{\varepsilon} \tff(t_0),
 \end{equation*}
whence \eqref{T:maximal_noncomp_alt_m_phi:capital_M_eq5} follows. At last, assume that \ref{item:possibility_capital_M_not_identity} is the case (in particular, $\mu(\RRR) < \infty$). Note that
\begin{equation}\label{T:maximal_noncomp_alt_m_phi:capital_M_eq7}
    \tff(t_0) = t_0 \max\Bigg\{\sup_{s\in[t_0, \frac{\mu(\RRR)}{N})} \frac{\sup_{\tau\in(0, s]} \varphi(\tau)}{s}, \sup_{s\in[\frac{\mu(\RRR)}{N}, \mu(\RRR))} \frac{\sup_{\tau\in(0, s]} \varphi(\tau)}{s} \Bigg\}.
\end{equation}
We claim that
\begin{equation}
    t_0 \sup_{s\in[\frac{\mu(\RRR)}{N}, \mu(\RRR))} \frac{\sup_{\tau\in(0, s]} \varphi(\tau)}{s} < \tff(t_0). \label{T:maximal_noncomp_alt_m_phi:capital_M_eq8}
\end{equation}
To this end, note that we have $t_0< a \leq \beta_0$ thanks to \eqref {T:maximal_noncomp_alt_m_phi:a_small_enough} and \eqref{T:maximal_noncomp_alt_m_phi:non_identity_beta0}. Hence, using this observation, \eqref{T:maximal_noncomp_alt_m_phi:non_identity_jump}, and the monotonicity of the functions $\tff$ and $t\mapsto \tff(t)/t$, we have
\begin{align*}
    t_0 \sup_{s\in[\frac{\mu(\RRR)}{N}, \mu(\RRR))} \frac{\sup_{\tau\in(0, s]} \varphi(\tau)}{s} &\leq t_0 \sup_{s\in[\frac{\mu(\RRR)}{N}, \mu(\RRR))} \frac{\tff(s)}{s} = t_0 \frac{\tff(\frac{\mu(\RRR)}{N})}{\frac{\mu(\RRR)}{N}} \\
    &< t_0 \frac{\tff(\beta_0)}{\beta_0} \leq t_0 \frac{\tff(t_0)}{t_0} = \tff(t_0). 
\end{align*}
Consequently, combining \eqref{T:maximal_noncomp_alt_m_phi:capital_M_eq7} with \eqref{T:maximal_noncomp_alt_m_phi:capital_M_eq8}, we have
\begin{equation*}
    \tff(t_0) = t_0 \sup_{s\in[t_0, \frac{\mu(\RRR)}{N})} \frac{\sup_{\tau\in(0, s]} \varphi(\tau)}{s}.
\end{equation*}
Finally, plugging this into \eqref{T:maximal_noncomp_alt_m_phi:capital_M_eq6}, we obtain \eqref{T:maximal_noncomp_alt_m_phi:capital_M_eq5}, which concludes the proof.
\end{proof}

\begin{remark}\label{rem:maximal_noncomp_alt_m_phi:phi_quasiconcave}\leavevmode
\begin{enumerate}[label={\rm(\roman*)},ref={\rm(\roman*)}]
    \item The proof of the preceding theorem becomes simpler if we assume that $\varphi$ is almost quasiconcave in the case~\ref{item:capital_M_alt}. In particular, we then do not need to distinguish among the three possibilities \ref{item:possibility_small_m_or_infinite_measure}--\ref{item:possibility_capital_M_identity}. We proceed the same way as when \ref{item:possibility_small_m_or_infinite_measure} is the case until we establish \eqref{T:maximal_noncomp_alt_m_phi:capital_M_eq4}. The only difference is that instead of \eqref{T:maximal_noncomp_alt_m_phi:not_too_constant_chosen_theta} we fix $N\in\N$ such that
\begin{equation}\label{rem:maximal_noncomp_alt_m_phi:phi_quasiconcave:eq}
    \inf_{t\in(0, \mu(R)/(N(1-\theta)))} \frac{\varphi(N(1-\theta) t)}{\varphi(t)} > \frac{2 C}{\varepsilon C_\varphi} \|T\| ,
\end{equation}
where $C_\varphi\in(0, 1]$ is the constant from \eqref{E:quasiconcave_majorant_equivalence}. Combining \eqref{rem:maximal_noncomp_alt_m_phi:phi_quasiconcave:eq} with \eqref{E:quasiconcave_majorant_equivalence}, we immediately obtain \eqref{T:maximal_noncomp_alt_m_phi:capital_M_eq5}, no matter which of the three possibilities is true. Having \eqref{T:maximal_noncomp_alt_m_phi:capital_M_eq5} at our disposal, we reach a contradiction the same way as before.
\item As a straightforward application of Theorem~\ref{T:maximal_noncomp_alt_m_phi}, we can recover the maximal noncompactness of the weak Sobolev-Lorentz embedding $I\colon V_0^{m} L^{p,q}(\Omega) \to L^{\frac{np}{n-mp}, \infty}(\Omega)$, which was obtained in~\cite{Lan:22}. Here, $\Omega\subseteq \rn$ is an open bounded set, either $p\in(1, \frac{n}{m})$ and $q\in[1, \infty]$ or $p=q=1$, and $V_0^m L^{p,q}(\Omega)$ is the homogeneous Sobolev-Lorentz space consisting of $m$-times weakly differentiable functions in $\Omega$, $1\leq m < n$, whose continuation by $0$ outside $\Omega$ is $m$-times weakly differentiable in $\rn$ and whose $m$-th order weak derivates belong to the Lorentz space $L^{p,q}(\Omega)$. The Sobolev-Lorentz space $V_0^{m} L^{p,q}(\Omega)$ is endowed with the (quasi)norm $\|f\|_{V_0^m L^{p,q}(\Omega)} = \||\nabla^m f|\|_{L^{p,q}(\Omega)}$, where $\nabla^m f$ is the vector of all $m$-th order weak derivatives of $f$. Recall that the Lorentz space $L^{p,q}(\Omega)$ consists of all measurable functions $f$ on $\Omega$ such that
\begin{equation*}
    \| f \|_{L^{p,q}(\Omega)} = \|t^{\frac1{p} - \frac1{q}} f^*(t)\|_{L^q(0, |\Omega|)} < \infty.
\end{equation*}
The weak Lebesgue space $L^{\frac{np}{n-mp}, \infty}(\Omega)$ is the Marcinkiewicz space $m_\varphi(\Omega)$ corresponding to $\varphi(t) = t^\frac{n-mp}{np}$, $t\in(0, |\Omega|)$, though it may also be regarded as a Lorentz space. Now, it can be readily verified that $\varphi$ satisfies the assumptions of Theorem~\ref{T:maximal_noncomp_alt_m_phi}. Given any $\lambda\in(0, \|I\|)$, where $\|I\|$ is the norm of the embedding $I\colon V_0^{m} L^{p,q}(\Omega) \to L^{\frac{np}{n-mp}, \infty}(\Omega)$, $a\in(0, \mu(\Omega))$, and $M\in\N$, we fix a collection $B_1, \dots, B_M\subseteq \Omega$ of mutually disjoint balls such that $|B_j|\leq a$ for every $j\in\{1, \dots, M\}$. It can be easily observed that
\begin{equation*}
    \frac{\|f_\kappa\|_{L^{\frac{np}{n-mp}, \infty}(\Omega)}}{\|f_\kappa\|_{V_0^{m} L^{p,q}(\Omega)}} = \frac{\|f\|_{L^{\frac{np}{n-mp}, \infty}(\Omega)}}{\|f\|_{V_0^{m} L^{p,q}(\Omega)}}
\end{equation*}
for every $f\in V_0^{m} L^{p,q}(\Omega)$, where $f_\kappa(x) = f(\kappa x)$, and that the support of $f_\kappa$ vanishes as $\kappa\to\infty$. It follows that we can find a collection of functions $\{f_j\}_{j = 1}^M \subseteq B_{V_0^{m} L^{p,q}(\Omega)}$ such that $\spt f_j \subseteq B_j$ and $\|f_j\|_{L^{\frac{np}{n-mp}, \infty}(\Omega)} \geq \lambda$ for every $j\in\{1, \dots, M\}$, and such that the functions $\{f_j\}_{j = 1}^M$ are translated copies of each other (in particular, they are equimeasurable). Hence, the embedding $I\colon V_0^{m} L^{p,q}(\Omega) \to L^{\frac{np}{n-mp}, \infty}(\Omega)$ is maximally noncompact. Finally, note that we may replace $m_\varphi(\Omega)$ with $M_\varphi(\Omega)$ and/or, when $p\in(1, \frac{n}{m})$, we may replace $L^{p,q}(\Omega)$ with $L^{(p,q)}(\Omega)$, a variant of the Lorentz space defined by means of $f^{**}$ instead of $f^*$. The maximal noncompactness of the modified weak Sobolev-Lorentz embedding then follows in the same way.
\end{enumerate}

\end{remark}

\section{Embeddings into the space of essentially bounded functions}
\label{S:embeddings-into-l-infty}

This section contains a theorem suitable for proving lower bounds on the measure of noncompactness of embeddings into $L^\infty(\RRR, \mu)$. Although $M_\varphi(\RRR, \mu) = m_\varphi(\RRR, \mu) = L^\infty(\RRR, \mu)$ for $\varphi\equiv1$, the situation is fundamentally different from that in the previous section. Theorem~\ref{T:max-noncompact} cannot be used when the target space is the space of essentially bounded functions. In particular, assumption~\eqref{E:varphi-vanishing-at-0}, which was important for the proof to work, is not satisfied. Nevertheless, the $L^\infty$ norm gives us a lot of information about the size of functions, which the following theorem exploits.

\begin{theorem} \label{T:beta-lower-of-X-to-ell-infty}
Let $X$ be a quasinormed space such that $X \subseteq L^\infty(\RRR, \mu)$. Let $I\colon X \to L^\infty(\RRR, \mu)$ be the embedding operator. Let $\rpar\in(0, \|I\|]$ and assume that for every $\ell\in\N$, there are functions $f_i\in B_X$, $i=1,\dots,\ell$, with pairwise disjoint supports satisfying
\begin{equation} \label{E:X-norm-difference}
	\nrm{f_i-f_j}_{X} \le 1
		\quad\text{for every $i\neq j\in\{1,\dots,\ell\}$}
\end{equation}
and
\begin{equation} \label{E:u-infty-norm-large}
	\nrm{f_i}_{L^\infty} > r
		\quad\text{for every $i\in\{1,\dots,\ell\}$}.
\end{equation}
Then
\begin{equation*}
	\bmc(I)\ge\rpar.
\end{equation*}
\end{theorem}

\begin{proof}
Suppose that $\bmc(I) < \rpar$. Consequently, there are $m\in\N$, $m\geq2$, and a collection $\{g_1,\dots,g_{m-1}\}\subseteq L^\infty(\RRR, \mu)$ such that
\begin{equation} \label{E:m-1-covering}
    I(B_{X})
			\subseteq\bigcup_{k=1}^{m-1}\left(g_k+\rpar B_{L^{\infty}}\right).
\end{equation}
Set $\ell=2^m$. Let $f_1,\dots,f_\ell\in B_X$ be the collection of functions whose existence is
guaranteed by the assumptions of the theorem, and set
\begin{equation*}
E_i=\{x\in\RRR\colon f_i\ne 0\},\ i\in\{1, \dots, \ell\}.
\end{equation*}
Next, define
\begin{equation*}
W_i=\{k\in \{1,\dots,m{-}1\}\colon \|f_i-g_k\|_{L^{\infty}(E_i)} \leq r\},\quad i\in\{1,\dots,\ell\},
\end{equation*}
and note that these sets are nonempty thanks to \eqref{E:m-1-covering}. Hence, by the pigeonhole principle, there are indices $i\ne j\in\{1, \dots, \ell\}$ such that $W_i=W_j$, thanks to the fact that $m-1 < \ell$. We fix such two distinct indices. Set $h=f_i-f_j$, and note that $\|h\|_{X}\le 1$ thanks to inclusion~\eqref{E:X-norm-difference}. Hence, there is $k\in \{1,\dots,m{-}1\}$ such that
\begin{equation}\label{E:T:beta-lower-of-X-to-ell-infty:w_close_to_vk}
\|h - g_k\|_{L^{\infty}(\RRR)} \leq r
\end{equation}
owing to \eqref{E:m-1-covering}. Furthermore, note that $h\chi_{E_i} = f_i$ and $h\chi_{E_j} = -f_j$ thanks to the fact that $E_i \cap E_j = \emptyset$. We claim that $k \in W_i$. Indeed, since $h\chi_{E_i} = f_i$, we have
\begin{equation*}
\|f_i-g_k\|_{L^{\infty}(E_i)}=\|h - g_k\|_{L^{\infty}(E_i)} \leq r.
\end{equation*}
Moreover, since $W_i = W_j$, we also have
\begin{equation*}
\|f_j-g_k\|_{L^{\infty}(E_j)} \leq r.
\end{equation*}

Now, using \eqref{E:u-infty-norm-large} and writing $2f_j = (f_j + g_k) + (f_j - g_k)$, we obtain
\begin{equation*}
2r<2\|f_j\|_{L^{\infty}(E_j)}\le \|f_j+g_k\|_{L^{\infty}(E_j)}+\|f_j-g_k\|_{L^{\infty}(E_j)} \leq \|f_j+g_k\|_{L^{\infty}(E_j)} + r,
\end{equation*}
whence
\begin{equation*}
\|f_j+g_k\|_{L^{\infty}(E_j)}>r.
\end{equation*}
However, since $h\chi_{E_j} = -f_j$, it follows from the last inequality that
\begin{equation*}
\|h-g_k\|_{L^{\infty}(\RRR)}\ge\|h - g_k\|_{L^{\infty}(E_j)}=\|f_j+g_k\|_{L^{\infty}(E_j)}>r,
\end{equation*}
which contradicts \eqref{E:T:beta-lower-of-X-to-ell-infty:w_close_to_vk}.
\end{proof}


\begin{thebibliography}{40}
\providecommand{\natexlab}[1]{#1}
\providecommand{\url}[1]{\texttt{#1}}
\expandafter\ifx\csname urlstyle\endcsname\relax
  \providecommand{\doi}[1]{doi: #1}\else
  \providecommand{\doi}{doi: \begingroup \urlstyle{rm}\Url}\fi

\bibitem[Bana\'{s} and Goebel(1980)]{Ban:80}
J.~Bana\'{s} and K.~Goebel.
\newblock \emph{Measures of noncompactness in {B}anach spaces}, volume~60 of
  \emph{Lecture Notes in Pure and Applied Mathematics}.
\newblock Marcel Dekker, Inc., New York, 1980.
\newblock ISBN 0-8247-1248-X.

\bibitem[Bennett and Sharpley(1988)]{Ben:88}
C.~Bennett and R.~Sharpley.
\newblock \emph{Interpolation of operators}, volume 129 of \emph{Pure and
  Applied Mathematics}.
\newblock Academic Press, Inc., Boston, MA, 1988.
\newblock ISBN 0-12-088730-4.

\bibitem[Bouchala(2020)]{Bou:19}
O.~Bouchala.
\newblock Measures of {N}on-{C}ompactness and {S}obolev--{L}orentz {S}paces.
\newblock \emph{Z. Anal. Anwend.}, 39\penalty0 (1):\penalty0 27--40, 2020.
\newblock ISSN 0232-2064.
\newblock \doi{10.4171/zaa/1649}.

\bibitem[Carl(1981)]{Car:81}
B.~Carl.
\newblock Entropy numbers, {$s$}-numbers, and eigenvalue problems.
\newblock \emph{J. Funct. Anal.}, 41\penalty0 (3):\penalty0 290--306, 1981.
\newblock ISSN 0022-1236.
\newblock \doi{10.1016/0022-1236(81)90076-8}.

\bibitem[Carl and Triebel(1980)]{Car:80}
B.~Carl and H.~Triebel.
\newblock Inequalities between eigenvalues, entropy numbers, and related
  quantities of compact operators in {B}anach spaces.
\newblock \emph{Math. Ann.}, 251\penalty0 (2):\penalty0 129--133, 1980.
\newblock ISSN 0025-5831.
\newblock \doi{10.1007/BF01536180}.

\bibitem[Cavaliere and Mihula(2019)]{Cav:19}
P.~Cavaliere and Z.~Mihula.
\newblock Compactness for {S}obolev-type trace operators.
\newblock \emph{Nonlinear Anal.}, 183:\penalty0 42--69, 2019.
\newblock ISSN 0362-546X.
\newblock \doi{10.1016/j.na.2019.01.013}.

\bibitem[Cianchi et~al.(2023)Cianchi, Musil, and Pick]{Cia:23}
A.~Cianchi, V.~Musil, and L.~Pick.
\newblock Optimal {S}obolev embeddings for the {O}rnstein-{U}hlenbeck operator.
\newblock \emph{J. Differential Equations}, 359:\penalty0 414--475, 2023.
\newblock ISSN 0022-0396.
\newblock \doi{10.1016/j.jde.2023.02.035}.

\bibitem[Cwikel and Pustylnik(2000)]{Cwi:00}
M.~Cwikel and E.~Pustylnik.
\newblock Weak type interpolation near ``endpoint'' spaces.
\newblock \emph{J. Funct. Anal.}, 171\penalty0 (2):\penalty0 235--277, 2000.
\newblock ISSN 0022-1236.
\newblock \doi{10.1006/jfan.1999.3502}.

\bibitem[Darbo(1955)]{Dar:55}
G.~Darbo.
\newblock Punti uniti in trasformazioni a codominio non compatto.
\newblock \emph{Rend. Sem. Mat. Univ. Padova}, 24:\penalty0 84--92, 1955.
\newblock ISSN 0041-8994.
\newblock URL \url{http://www.numdam.org/item?id=RSMUP_1955__24__84_0}.

\bibitem[Edmunds and Evans(2018)]{Edm:18}
D.~E. Edmunds and W.~D. Evans.
\newblock \emph{Spectral theory and differential operators}.
\newblock Oxford Mathematical Monographs. Oxford University Press, Oxford,
  2018.
\newblock ISBN 978-0-19-881205-0.
\newblock \doi{10.1093/oso/9780198812050.001.0001}.

\bibitem[Fern\'{a}ndez-Mart\'{\i}nez et~al.(2010)Fern\'{a}ndez-Mart\'{\i}nez,
  Manzano, and Pustylnik]{Fer:10}
P.~Fern\'{a}ndez-Mart\'{\i}nez, A.~Manzano, and E.~Pustylnik.
\newblock Absolutely continuous embeddings of rearrangement-invariant spaces.
\newblock \emph{Mediterr. J. Math.}, 7\penalty0 (4):\penalty0 539--552, 2010.
\newblock ISSN 1660-5446.
\newblock \doi{10.1007/s00009-010-0039-y}.

\bibitem[Gogatishvili and Pick(2006)]{GP:06}
A.~Gogatishvili and L.~Pick.
\newblock Embeddings and duality theorems for weak classical {L}orentz spaces.
\newblock \emph{Canad. Math. Bull.}, 49\penalty0 (1):\penalty0 82--95, 2006.
\newblock ISSN 0008-4395,1496-4287.
\newblock \doi{10.4153/CMB-2006-008-3}.

\bibitem[Hencl(2003)]{Hen:03}
S.~Hencl.
\newblock Measures of non-compactness of classical embeddings of {S}obolev
  spaces.
\newblock \emph{Math. Nachr.}, 258:\penalty0 28--43, 2003.
\newblock ISSN 0025-584X.
\newblock \doi{10.1002/mana.200310085}.

\bibitem[Hunt(1966)]{Hun:66}
R.~A. Hunt.
\newblock On {$L(p,\,q)$} spaces.
\newblock \emph{Enseign. Math. (2)}, 12:\penalty0 249--276, 1966.
\newblock ISSN 0013-8584.

\bibitem[Kalton and Kami\'{n}ska(2005)]{KK:05}
N.~J. Kalton and A.~Kami\'{n}ska.
\newblock Type and order convexity of {M}arcinkiewicz and {L}orentz spaces and
  applications.
\newblock \emph{Glasg. Math. J.}, 47\penalty0 (1):\penalty0 123--137, 2005.
\newblock ISSN 0017-0895,1469-509X.
\newblock \doi{10.1017/S0017089504002204}.

\bibitem[Kerman and Pick(2008)]{Ker:08}
R.~Kerman and L.~Pick.
\newblock Compactness of {S}obolev imbeddings involving rearrangement-invariant
  norms.
\newblock \emph{Studia Math.}, 186\penalty0 (2):\penalty0 127--160, 2008.
\newblock ISSN 0039-3223.
\newblock \doi{10.4064/sm186-2-2}.

\bibitem[Kre\u{\i}n and Semenov(1972)]{Kre:72}
S.~G. Kre\u{\i}n and E.~M. Semenov.
\newblock A certain property of equimeasurable functions.
\newblock \emph{Vorone\v{z}. Gos. Univ. Trudy Nau\v{c}n.-Issled. Inst. Mat.
  VGU}, \penalty0 (5, Sb. State\u{\i} Teor. Operator. i Differencial.
  Uravneni\u{\i}):\penalty0 70--74, 1972.

\bibitem[Kre\u{\i}n et~al.(1982)Kre\u{\i}n, Petun\={\i}n, and
  Sem\"{e}nov]{Kre:82}
S.~G. Kre\u{\i}n, Y.~I. Petun\={\i}n, and E.~M. Sem\"{e}nov.
\newblock \emph{Interpolation of linear operators}, volume~54 of
  \emph{Translations of Mathematical Monographs}.
\newblock American Mathematical Society, Providence, RI, 1982.
\newblock ISBN 0-8218-4505-7.
\newblock Translated from the Russian by J. Sz\H{u}cs.

\bibitem[Kuratowski(1930)]{Kur:30}
C.~Kuratowski.
\newblock Sur les espaces complets.
\newblock \emph{Fund. Math.}, 15:\penalty0 301--309, 1930.
\newblock ISSN 0016-2736.
\newblock \doi{10.4064/fm-15-1-301-309}.

\bibitem[Lang and Musil(2019)]{Lan:19}
J.~Lang and V.~Musil.
\newblock Strict {$s$}-numbers of non-compact {S}obolev embeddings into
  continuous functions.
\newblock \emph{Constr. Approx.}, 50\penalty0 (2):\penalty0 271--291, 2019.
\newblock ISSN 0176-4276,1432-0940.
\newblock \doi{10.1007/s00365-018-9448-0}.

\bibitem[Lang et~al.(2021)Lang, Musil, Ol\v{s}\'{a}k, and Pick]{Lan:20}
J.~Lang, V.~Musil, M.~Ol\v{s}\'{a}k, and L.~Pick.
\newblock Maximal non-compactness of {S}obolev embeddings.
\newblock \emph{J. Geom. Anal.}, 31\penalty0 (9):\penalty0 9406--9431, 2021.
\newblock ISSN 1050-6926,1559-002X.
\newblock \doi{10.1007/s12220-020-00522-y}.

\bibitem[Lang et~al.(2022)Lang, Mihula, and Pick]{Lan:22}
J.~Lang, Z.~Mihula, and L.~Pick.
\newblock Compactness of {S}obolev embeddings and decay of norms.
\newblock \emph{Studia Math.}, 265\penalty0 (1):\penalty0 1--35, 2022.
\newblock ISSN 0039-3223.
\newblock \doi{10.4064/sm201119-29-9}.

\bibitem[Lang et~al.(2024)Lang, Mihula, and Pick]{LMP:preprint}
J.~Lang, Z.~Mihula, and L.~Pick.
\newblock Maximal noncompactness of limiting {S}obolev embeddings.
\newblock \emph{Proc. Roy. Soc. Edinburgh Sect. A}, First View:\penalty0 19
  pp., 2024.
\newblock ISSN 0308-2105,1473-7124.
\newblock \doi{10.1017/prm.2024.93}.

\bibitem[Lindenstrauss and Tzafriri(1979)]{LT:79}
J.~Lindenstrauss and L.~Tzafriri.
\newblock \emph{Classical {B}anach spaces. {II}}, volume~97 of \emph{Ergebnisse
  der Mathematik und ihrer Grenzgebiete [Results in Mathematics and Related
  Areas]}.
\newblock Springer-Verlag, Berlin-New York, 1979.
\newblock ISBN 3-540-08888-1.
\newblock Function spaces.

\bibitem[Lorentz(1951)]{Lor:51}
G.~G. Lorentz.
\newblock On the theory of spaces {$\Lambda$}.
\newblock \emph{Pacific J. Math.}, 1:\penalty0 411--429, 1951.
\newblock ISSN 0030-8730.
\newblock URL \url{http://projecteuclid.org/euclid.pjm/1103052109}.

\bibitem[Maz'ya(2011)]{Maz:11}
V.~Maz'ya.
\newblock \emph{Sobolev spaces with applications to elliptic partial
  differential equations}, volume 342 of \emph{Grundlehren der Mathematischen
  Wissenschaften [Fundamental Principles of Mathematical Sciences]}.
\newblock Springer, Heidelberg, augmented edition, 2011.
\newblock ISBN 978-3-642-15563-5.
\newblock \doi{10.1007/978-3-642-15564-2}.

\bibitem[Nekvinda and Pe\v{s}a(2024)]{Nek:24}
A.~Nekvinda and D.~Pe\v{s}a.
\newblock On the properties of quasi-{B}anach function spaces.
\newblock \emph{J. Geom. Anal.}, 34\penalty0 (8):\penalty0 Paper No. 231, 29
  pages, 2024.
\newblock ISSN 1050-6926,1559-002X.
\newblock \doi{10.1007/s12220-024-01673-y}.

\bibitem[O'Neil(1968)]{One:68}
R.~O'Neil.
\newblock Integral transforms and tensor products on {O}rlicz spaces and
  ${L}(p,q)$ spaces.
\newblock \emph{J. Analyse Math.}, 21:\penalty0 1--276, 1968.
\newblock \doi{10.1007/BF02787670}.

\bibitem[Opic and Pick(1999)]{OP:99}
B.~Opic and L.~Pick.
\newblock On generalized {L}orentz-{Z}ygmund spaces.
\newblock \emph{Math. Inequal. Appl.}, 2\penalty0 (3):\penalty0 391--467, 1999.
\newblock ISSN 1331-4343,1848-9966.
\newblock \doi{10.7153/mia-02-35}.

\bibitem[Peetre(1966)]{Pee:66}
J.~Peetre.
\newblock Espaces d'interpolation et th\'{e}or\`eme de {S}oboleff.
\newblock \emph{Ann. Inst. Fourier (Grenoble)}, 16\penalty0 (fasc. 1):\penalty0
  279--317, 1966.
\newblock ISSN 0373-0956.
\newblock \doi{10.5802/aif.232}.

\bibitem[Pohozhaev(1965)]{Poh:65}
S.~I. Pohozhaev.
\newblock On the imbedding {S}obolev theorem for $pl=n$.
\newblock \emph{Doklady Conference, Section Math. Moscow Power Inst.},
  165:\penalty0 158--170 (Russian), 1965.

\bibitem[Sadovski\u{\i}(1968)]{Sad:68}
B.~N. Sadovski\u{\i}.
\newblock Measures of noncompactness and condensing operators.
\newblock \emph{Problemy Mat. Anal. Slo\v{z}. Sistem}, \penalty0 (2):\penalty0
  89--119, 1968.
\newblock ISSN 0137-6675.

\bibitem[Semenov(1963)]{Sem:63}
E.~M. Semenov.
\newblock A scale of spaces with an interpolation property.
\newblock \emph{Dokl. Akad. Nauk SSSR}, 148:\penalty0 1038--1041, 1963.
\newblock ISSN 0002-3264.

\bibitem[Semenov(1964)]{Sem:64}
E.~M. Semenov.
\newblock Embedding theorems for {B}anach spaces of measurable functions.
\newblock \emph{Dokl. Akad. Nauk SSSR}, 156:\penalty0 1292--1295, 1964.
\newblock ISSN 0002-3264.

\bibitem[Sharpley(1972)]{Sha:72}
R.~Sharpley.
\newblock Spaces {$\Lambda _{\alpha }(X)$} and interpolation.
\newblock \emph{J. Functional Analysis}, 11:\penalty0 479--513, 1972.
\newblock ISSN 0022-1236.
\newblock \doi{10.1016/0022-1236(72)90068-7}.

\bibitem[Slav\'{\i}kov\'{a}(2012)]{Sla:12}
L.~Slav\'{\i}kov\'{a}.
\newblock Almost-compact embeddings.
\newblock \emph{Math. Nachr.}, 285\penalty0 (11-12):\penalty0 1500--1516, 2012.
\newblock ISSN 0025-584X.
\newblock \doi{10.1002/mana.201100286}.

\bibitem[Slav\'{\i}kov\'{a}(2015)]{Sla:15}
L.~Slav\'{\i}kov\'{a}.
\newblock Compactness of higher-order {S}obolev embeddings.
\newblock \emph{Publ. Mat.}, 59\penalty0 (2):\penalty0 373--448, 2015.
\newblock ISSN 0214-1493.
\newblock URL \url{http://projecteuclid.org/euclid.pm/1438261121}.

\bibitem[Stein(1981)]{Ste:81}
E.~M. Stein.
\newblock Editor's note: the differentiability of functions in {${\bf R}^{n}$}.
\newblock \emph{Ann. of Math. (2)}, 113\penalty0 (2):\penalty0 383--385, 1981.
\newblock ISSN 0003-486X.
\newblock URL \url{https://www.jstor.org/stable/2006989}.

\bibitem[Trudinger(1967)]{Tru:67}
N.~S. Trudinger.
\newblock On imbeddings into {O}rlicz spaces and some applications.
\newblock \emph{J. Math. Mech.}, 17:\penalty0 473--483, 1967.
\newblock \doi{10.1512/iumj.1968.17.17028}.

\bibitem[Yudovich(1961)]{Yud:61}
V.~I. Yudovich.
\newblock Some estimates connected with integral operators and with solutions
  of elliptic equations.
\newblock \emph{Soviet Math. Doklady}, 2:\penalty0 746--749 (Russian), 1961.

\end{thebibliography}
\end{document}